\documentclass[10pt]{amsart}
\usepackage{amsmath,amssymb,verbatim}
\usepackage[utf8]{inputenc}
\usepackage{enumerate}
\usepackage{pdflscape}
\usepackage{amsthm}
\usepackage{mathrsfs}
\usepackage{color}
\usepackage[normalem]{ulem}
\usepackage{cancel}
\usepackage{tikz}
\usetikzlibrary{matrix}
\usepackage[all]{xy}
\usepackage{mathtools}
\usepackage{cleveref}
\usepackage{arydshln}
\usepackage[shortlabels]{enumitem}

\makeatletter
\@namedef{subjclassname@1991}{\textup{2020} Mathematics Subject Classification}
\makeatother

\topskip=\baselineskip 
\newtheorem{theorem}{Theorem}[section]
\newtheorem{lemma}[theorem]{Lemma}
\newtheorem{proposition}[theorem]{Proposition}

\newtheorem{corollary}[theorem]{Corollary}
\newtheorem{remark}[theorem]{Remark}

\newtheorem*{theorem*}{Theorem}

\makeatletter
\@addtoreset{equation}{section}
\makeatother

\newcommand{\bb}{{\overline b}}

\newcommand{\Aut}{\mbox{\rm Aut}}

\newcommand{\inv}{\text{inv}}

\newcommand{\Z}{{\mathbb Z}}

\newcommand{\B}{\mathcal{B}}

\newcommand{\matriz}[1]{\begin{array} #1 \end{array}}

\newcommand{\GEN}[1]{\left\langle #1 \right\rangle}

\newcommand{\U}{\mathcal{U}}

\newcommand{\e}[1]{\left\lfloor #1 \right\rfloor}

\newcommand{\ord}{\text{o}}
\newcommand{\lex}{\text{lex}}

\newcommand{\Ese}[2]{\mathcal{S}\left(#1\mid #2\right)}
\newcommand{\Te}[2]{\mathcal{T}\left(#1\mid #2\right)}

\newcommand{\qand}{\quad \text{and} \quad}

\title[A classification of the finite two-generated $p$-groups with cyclic derived subgroup]{A classification of the finite two-generated cyclic-by-abelian groups of prime-power order}

\author{Osnel Broche, Diego Garc\'{\i}a-Lucas and \'{A}ngel del R\'{\i}o}
\thanks{The first author has been partially supported by Fundación Séneca of Murcia under a Jiménez de la Espada grant 20598/IV/18.
The third author has been partially supported by Grant 19880/GERM/15 funded by Fundaci\'{o}n S\'{e}neca of Murcia. The second and third authors have been partially supported by Grant PID2020-113206GB-I00 funded by MCIN/AEI/10.13039/501100011033.}

\address{O. Broche: Departamento de Matemática e Matemática Aplicada, Universidade Federal de Lavras, Caixa Postal 3037, 37200-000, Lavras, Brazil. \rm{osnel@ufla.br}}

\address{D. García-Lucas, \'{A}. del R\'{i}o: Departamento de Matem\'{a}ticas, Universidad de Murcia, 30100, Murcia, Spain. \rm{diego.garcial@um.es, adelrio@um.es}}



\begin{document}
\maketitle

\begin{abstract}
We obtain a classification of the finite two-generated cyclic-by-abelian groups of prime-power order. We associate to each such group $G$ a list
$\inv(G)$ of numerical group invariants which determines the isomorphism type of $G$. Then we describe the set formed by all the possible values of $\inv(G)$.
This allows us to develop practical algorithms to construct all finite two-generated cyclic-by-abelian groups of a given prime-power order, to compute the invariants of such a group,  and to decide whether two such groups are isomorphic.
\end{abstract}

\section{Introduction}

Classifying groups up to isomorphism is a fundamental problem in Group Theory which was already identified in the seminal work of Cayley on finite groups \cite{Cayley1878} where he wrote: ``The general problem is to find all the groups of a given order''.
Unfortunately, an  answer to this question is far from being attainable unless one restricts to particularly well behaved groups such as, for example, abelian finitely generated groups, or finite metacyclic groups (see e.g. \cite{Hempel2000,GarciadelRio2023} for the latter case).
The special case of groups of prime-power order is particularly difficult as its was observed by P. Hall in \cite[page 131]{Hall1940} where he wrote: ``To put it crudely, there is no apparent limit to the complication of a prime-power group. [...] And its seems unlikely that it will be possible to compass the overwhelming variety of prime-power groups within the bounds of a single finite system of formulae''.  This is illustrated, for example, by the 33 pages that Blackburn required to classify the finite $p$-groups with derived subgroup of order $p$ \cite{Blackburn1999}.
A different approach aims at a classification of the $p$-groups of a given order. This is completed up to $p^7$, for $p$ odd, and up to order $2^9$ \cite{OBrienVaughan2005,EickOBrien}.

Besides the basic interest in classifying, up to isomorphism, the groups of a particular type, it is of fundamental importance to deal with other questions. Our initial motivation was trying to solve the Modular Isomorphism Problem for cyclic-by-abelian two-generated $p$-groups and yet the preparation of this paper paved the way to find a negative solution for this problem \cite{GarciaMargolisdelRio} for $p=2$ and has been essential in obtaining some positive results for $p>2$ \cite{GarciadelRioStanojkowski}.

A classification of the  non-abelian  cyclic-by-abelian 2-generated $p$-groups is only available in the literature if $p$ is odd \cite{Miech1975, Song2013} or if the groups are assumed to be of class 2 \cite{AMM}.
The aim of this paper is to fill this gap. 
More precisely we give a complete classification of such groups up to isomorphism, by associating to such a group $G$ a tuple  of integers 
$$\inv(G)=(p,m,n_1,n_2,\sigma_1,\sigma_2,o_1,o_2,o'_1,o'_2,u_1,u_2)$$ 
such that if $H$ is another such group then $G\cong H$ if and only if $\inv(G)=\inv(H)$, and describe the possible values of $\inv(G)$. 
As the classification is known for $p$ odd, the reader may wonder why we do not restrict
our treatment to the case $p=2$. There is no reduction of complexity by  considering only the case $p=2$, and hence for completeness we prefer to present the results in the general case. We followed the approach of Miech because it adapts better to the application we had in mind, namely the Modular Isomorphism Problem. Along the way we fix mistakes in Miech's classification (see \Cref{ErrorMiech}).
While the Miech and Song classifications split in various different presentations depending on parameters with no obvious group theoretical interpretation, we present a unified presentation for the group $G$ in terms of the entries of $\inv(G)$ (see \eqref{Presentacion}, \eqref{def-erres} and \eqref{def-tes}) and the group theoretical   role of each entry of $\inv(G)$ is clear from the definition. This provides an algorithmic procedure to compute all the groups under consideration and  to implement their construction in {\sf GAP} \cite{GAP4}. It also allows us to compute the invariants associated to a given group and hence to decide if two such groups are isomorphic.
We have implemented this and, with the help of the  {\sf GAP} package ANUPQ \cite{ANUPQ}, we have verified that our results agree with the output of the $p$-group generation algorithm \cite{OBrien1990} up to orders  $2^{12}, 3^{11}, 5^{10}, 7^9, 11^8, 13^7$ and $23^8$.

\medskip
To present our main result we need to fix some notation, and at the same time we outline the strategy we will follow. 
Let $G$ be a 2-generated non-abelian cyclic-by-abelian finite group of prime-power order. 
By the Burnside Basis Theorem \cite[5.3.2]{Robinson1982}, $G/G'$ is  $2$-generated and non-cyclic, and the first four invariants $p$, $m$, $n_1$ and $n_2$ of $G$ are given by 
$$|G'|=p^m \qand G/G'\cong C_{p^{n_1}}\times C_{p^{n_2}}, \quad \text{ with } n_1\ge n_2.$$
A \emph{basis} of $G$ is an ordered pair $b=(b_1,b_2)$ of elements of $G$ satisfying
	$$G/G'=\GEN{b_1G'} \times \GEN{b_2G'} \qand |b_iG'|=p^{n_i}\; (i=1,2).$$
Let $\B$ denote the set of bases of $G$.
Each basis determines a list of eight integers, and our strategy consists in selecting bases so that the associated lists satisfy an extreme condition with respect to a well order. This provides the additional eight entries of the list $\inv(G)$.
To define the integers associated to a basis we first define two maps $\sigma:G\rightarrow \{1,-1\}$ and $o:G\rightarrow \{0,1,\dots,m-1\}$ by setting:
	\begin{eqnarray*}
	\sigma(g)&=& \begin{cases} -1, & \text{if } a^g=a^{-1}\ne a \text{ for some } a\in G'; \\
	1, & \text{otherwise}.\end{cases} \\
	o(g) &=& \begin{cases} 0, & \text{if } a^g=a^{-1} \text{ for every } a\in G'; \\
	\log_p |gC_G(G')|, & \text{otherwise}.\end{cases}
	\end{eqnarray*}
So each basis $(b_1,b_2)$ of $G$ yields four integers $\sigma(b_i)$ and $o(b_i)$, $i=1,2$ and we use this to define the next four entries of $\inv(G)$ by setting 
$$\sigma o=(\sigma_1,\sigma_2,o_1,o_2) = 
\min_{\lex} \{(\sigma(b_1),\sigma(b_2),o(b_1),o(b_2)): (b_1,b_2)\in \B\}$$ 
where $  \displaystyle\min_{\lex}$ denotes  the minimum with respect to the lexicographical order. 
Let $r_1$ and $r_2$ be the unique integers $1<r_i\leq 1+p^m$ satisfying
\begin{equation} \label{def-erres}
r_1 \equiv  \sigma_1(1 +p^{m-o_1} ) \bmod p^m \qand
\begin{cases} 
r_2 \equiv  \sigma_2(1  + p^{m-o_2}) \bmod p^m,& \text{if } o_1o_2=0  ; \\
r_2 \equiv \sigma_2(1+p^{m-o_1})^{p^{o_1-o_2}} \bmod p^m  , & \text{otherwise.}
\end{cases} 
\end{equation}
Observe that the classes modulo $p^m$ represented by $r_i$ and $\sigma_i +p^{m-o_i}$ generate the same subgroup in the group of units of $\Z/p^m\Z$.

Let
$$\B_r=\{(b_1,b_2)\in \B :a^{b_i}=a^{r_i} \text{ for every } i=1,2 \text{ and } a\in G'\}.$$
In Theorem~\ref{Fijando-r}, we prove that $\B_r$ is not empty
From now on we only use bases in $\B_r$ and for each $b=(b_1,b_2)\in \B_r$ we denote by $t_1(b)$ and $t_2(b)$ the unique integers satisfying
\begin{equation}\label{Def-TesFunciones}
1\le t_i(b)\le p^m \qand b_i^{p^{n_i}}=[b_2,b_1]^{t_i(b)} \quad (i=1,2).
\end{equation}
Define $o'(b)=(o'_1(b),o'_2(b))$ and $u(b)=(u_2(b),u_1(b))$ by setting
	$$o'_i(b)=\log_p(|b_i|)-n_i \qand t_i(b)=u_i(b)p^{m-o'_i(b)}.$$
Observe that $|b_i|=p^{n_i+o'_i(b)}$ and hence $0\le o'_i(b)\le m$ and $p\nmid u_i(b)$.
We use this to define the next two entries of $\inv(G)$ by setting 
	$$(o'_1,o'_2)=\max_{\lex} \{ o'(b) :b\in \B_r \}.$$ 
Then we define
	$$\B_r' = \{b\in \B_r : o'(b)=(o'_1,o'_2)\}.$$
The two remaining entries of $\inv(G)$ are given by  
	$$(u_2,u_1)=\min_{\lex}  \{ u(b) : b\in \B_r'\}.$$
The ``unnatural'' order on the $u$'s is not a typo but a convenient technicality.
Observe that we are abusing notation since $o'_i,u_i$ and $t_i$ sometimes denote functions and sometimes integers related to those functions. This does not cause confusion because in the former case the functions always appear with arguments.

Set
\begin{equation}\label{def-tes}
t_i=u_ip^{m-o'_i} \quad (i=1,2).
\end{equation}
Now $G$ is isomorphic to $\mathcal G_I$, where $I$ is an abbreviation for $(p,m,n_1,n_2,\sigma_1,\sigma_2, o_1,o_2,o'_1,o'_2,u_1,u_2)$ and
\begin{equation}\label{Presentacion}
\mathcal{G}_I = \GEN{b_1,b_2 \mid [b_2,b_1]^{p^m}=1, \quad [b_2,b_1]^{b_i} = [b_2,b_1]^{r_i},\quad  b_i^{p^{n_i}} = [b_2,b_1]^{t_i}, \quad (i=1,2)},
\end{equation}
where $r_i$ and $t_i$ are as defined in \eqref{def-erres} and \eqref{def-tes}.
 
Hence, $G$ is completely determined up to isomorphism by $\inv(G)$. Therefore, to obtain our classification it only
remains to give the list of tuples occurring as $\inv(G)$.

\begin{theorem*}\label{Main}
The maps $[G]\mapsto \inv(G)$ and $I\mapsto [\mathcal{G}_I]$ define mutually inverse bijections between the isomorphism classes of $2$-generated non-abelian groups of prime-power order and the set of lists of integers $(p,m,n_1,n_2,\sigma_1,\sigma_2, o_1,o_2,o'_1,o'_2,u_1,u_2)$ satisfying the following conditions. 	
\begin{enumerate}
	\item \label{1} $p$ is prime and $n_1\geq n_2 \ge 1 $.
	\item \label{2}$\sigma_i=\pm 1$, $0\le o_i<\min(m,n_i)$  and $p\nmid u_i$ for $i=1,2$.   
	\item \label{3}If $p=2$ and $m\ge 2$ then $o_i<m-1$ for $i=1,2$.
	\item \label{4}$0\le o'_i \le m-o_i$ for $i=1,2$ and $o'_1\le m-o_2$.
 	\item \label{5} One of the following conditions holds:  
		\begin{enumerate}
		\item $o_1=0$.
		\item $0<o_1=o_2$ and $\sigma_2=-1$.
		\item $o_2=0<o_1$ and $n_2<n_1$.
		\item $0<o_2< o_1<o_2+n_1-n_2$.
		\end{enumerate}	
	
	\item \label{6}  Suppose that  $\sigma_1=1$. Then the following conditions hold: 
		\begin{enumerate}
			\item \label{6a} $\sigma_2=1$ and $ o_2+o_1'\leq m\le n_1$. 
			\item \label{6b} Either $o_1+o'_2\le m \le n_2$ or $2m-o_1-o'_2=n_2<m$ and $u_2\equiv 1 \bmod p^{m-n_2}$.
			\item \label{6c} If $o_1=0$ then one of the following conditions holds:
			\begin{enumerate}
				\item $o'_1\le o'_2\le o'_1+o_2+n_1-n_2$ and 
				$\max(p-2,o'_2 ,n_1-m)>0$. 
				\item $p=2$, $m=n_1$, $o'_2=0$ and $o_1'=1$. 
			\end{enumerate}
			\item \label{6d}If $o_2=0<o_1$ then $o'_1+\min(0,n_1-n_ 2  -o_1)\le o'_2\le o'_1+n_1-n_2$ and 
			$\max(p-2,o'_1 ,n_1-m)>0$.
			\item \label{6e}If $0<o_2<o_1$ then $o'_1\le o'_2\le o'_1+n_1-n_2$. 
			\item \label{6f}$ 1\le u_1  \leq p^{a_1}$, where 
			$$a_1=\min(o'_1,o_2,o_2+n_1-n_2+o'_1-o'_2).$$
			\item \label{6g}One of the following conditions holds:
			\begin{enumerate}
				\item $ 1\le  u_2 \leq p^{a_2}$.
				\item $o_1o_2\neq 0$, $n_1-n_2+o_1'-o_2'=0<a_1$,  $1+p^{a_2}\leq u_2 \leq 2p^{a_2}$, and $u_1 \equiv 1\bmod p$;
			\end{enumerate}
			where  
	$$a_2=  \begin{cases}
	0, &\text{if } o_1=0; \\
	\min(o_1,o'_2,o'_2-o'_1+\max(0,o_1+n_2-n_1)), & \text{if } o_2=0<o_1; \\ \min(o_1-o_2,o'_2-o'_1), & \text{otherwise.} \end{cases}$$
		\end{enumerate}	
	\item \label{7}  Suppose that  $\sigma_1=-1$. Then the following conditions hold: 
		\begin{enumerate}
			\item \label{7a} $p=2$, $m\ge 2$, $o'_1\le 1$ and $u_1=1$. 
			\item\label{T-+}\label{7b} If $\sigma_2=1$ then $n_2<n_1$ and the following conditions hold:
				\begin{enumerate}
					\item\label{T-+m>=}\label{7bi} If $m\leq n_2$ then $o'_2\leq 1$, $u_2=1$ and either  $o'_1 \le o'_2 $ or $o_2=0<n_1-n_2<o_1$
					
					\item\label{T-+m<}\label{7bii} If $m>n_2$ then $m+1=n_2+o'_2$, $u_2( 1+2^{m-o_1-1})\equiv -1 \bmod 2^{m-n_2}$, $1\le u_2 \le 2^{m-n_2+1}$, either $o'_1  =1$ or $   o_1+1\neq n_1$, and one of the following conditions holds:
					\begin{itemize}
						\item $o'_1=0$ and either $o_1=0$ or $o_2+1\ne n_2$.  
						\item $o_1'=1$, $o_2=0 $ and $ n_1-n_2<o_1$. 
						\item $  u_2\leq  2^{m-n_2}$.	
					\end{itemize}
				\end{enumerate} 
			
			\item\label{T--}\label{7c} If $\sigma_2=-1$ then $o'_2\le 1$, $u_2=1$ and the following conditions hold:
				\begin{enumerate}
					\item If  $o_1 \le o_2$  and $n_1>n_2$ then $o_1'  \leq o_2' $.
					\item If  $o_1=o_2$ and $n_1=n_2$ then $o_1' \geq o_2' $ 
					\item If $o_2=0<o_1=n_1-1$ and $n_2=1$ then  $o'_1 =1$ or $o'_2 =1$.  
					\item If $o_2=0<o_1$ and $n_1\ne o_1+1$ or $n_2\ne 1$ then $o'_1+\min(0,n_1-n_2-o_1)\le o'_2 $. 
					\item If $o_1 o_2 \ne 0$ and $o_1\ne o_2$ then $o'_1 \le o'_2  $. 
				\end{enumerate} 
		\end{enumerate}			
					
\end{enumerate}
\end{theorem*}

We explain now the structure of the paper.  
In \Cref{SectionBr} we prove that $\B_r\ne\emptyset$ and obtain some constraints for $\sigma o$  and $o'(b)$ for $b\in \B_r$. 
In \Cref{SectionChangingBases} we obtain formulae  that control $o'(b)$ and $u(b)$ for different bases in $\B_r$. This is used in \Cref{SectionBr'} to describe, only in terms of $o'(b)$, when a basis $b$ in $\B_r$ belongs to $\B'_r$. Then, in \Cref{SectionBrt}, we describe, only in terms of $u(b)$, when a basis $b$ in $\B'_r$ belongs to the set
	$$\B_{rt}=   \{b\in \B'_r : u(b)=(u_2,u_1)\}.$$ 
Combining these results we prove the theorem in \Cref{SeccionDemostracion}. 
 In \Cref{SectionImplementation} we present an implementation of our results and some experiments supporting the correctness of the main result.
In the appendix we collect technical number theoretical results used frequently in the proofs in \Cref{SectionBr} and \Cref{SectionChangingBases}.

\section{Fixing the $r_i$'s and constraints on the invariants}\label{SectionBr}

In this section we obtain some restrictions on the invariants of our target groups and we prove the existence of a basis $(b_1,b_2)$ such that $[b_2,b_1]^{b_i}=[b_2,b_1]^{r_i}$, where $r_i$ is defined by \eqref{def-erres}.

We start with some notation.
If $p$ is a prime integer and $n$ is a non-zero integer then $v_p(n)$ denotes the  maximum positive integer  $m$ with $p^m\mid n$. We set $v_p(0)=\infty$. 
If $m$ is a positive integer coprime to $n$ then $\ord_m(n)$ denotes the multiplicative order of $n$ modulo $m$, i.e. the minimum positive integer $k$ such that $n^k\equiv 1 \bmod n$.
We use $\le_{\lex}$ to denote the lexicographic order on lists of integers of the same length and $\min_{\lex}$ and $\max_{\lex}$ denote the minimum and maximum with respect to $\le_{\lex}$, respectively.

We use standard group theoretical notation. For example, the cyclic group  of order $n$ is denoted $C_n$ and if $G$ is a group then $G'$ denotes its derived subgroup. For  $g,h\in G$ we denote:
$$g^h = h^{-1}gh, \quad [g,h]=g^{-1}h^{-1}gh, \quad |g|=\text{ order of } g.$$
If $G'$ is cyclic and $g\in G$ then $r(g)$ denotes an integer, unique modulo $|G'|$, such that $a^g=a^{r(g)}$ for every $a\in G'$.

Given integers $s,t$ and $n$ with $n  \ge   0$ we set
$$\Ese{s}{n} = \sum_{i=0}^{n-1} s^i \qand \Te{s,t}{n} = \sum_{0\le  	i < j < n} s^i t^j.$$
This notation is motivated by the following lemma whose proof is straightforward. More properties of   these operators are  included  in \Cref{Apendice}.

\begin{lemma}\label{CyclicByAbelian}
If $G$ is a cyclic-by-abelian group then the following equalities hold:
	\begin{align}
	\label{CBAConmutadores} [x_1\cdots x_n,y_1\cdots y_m] &= 
	\prod_{i=1}^n \prod_{j=1}^m [x_i,y_j]^{x_{i+1}\cdots x_n y_{j+1}\cdots y_m} \quad & (x_1,\dots,x_n, y_1,\dots,y_m\in G),\\
	\label{CBAPotenciaga} (ga)^n &= g^n a^{\Ese{r(g)}{n}} \quad & (g\in G, a\in G'), \\
	\label{CBAPotenciagh} (gh)^n &= g^n h^n [h,g]^{\Te{r(g),r(h)}{n}} \quad
	& (g,h\in G).
	\end{align}
\end{lemma}

In the remainder of the paper $G$ is a 2-generated non-abelian cyclic-by-abelian group of prime-power order and
$$\inv(G)=(p,m,n_1,n_2,\sigma_1,\sigma_2,o_1,o_2,o'_1,o'_2,u_1,u_2).$$ 
Observe that 
\begin{equation}\label{sigmag=-1}
\sigma(g)=-1 \quad \text{if and only if} \quad p=2, m\ge 2 \text{ and  
} r(g)\equiv -1 \bmod 4.
\end{equation}
As $r(g)$ is coprime to $p$ and uniquely determined modulo $p^m$ we abuse notation by identifying $r(g)$ with an element of $\U_{p^m}$,  the  group of units of $\Z/p^m\Z$, and use standard group theoretical notation  for the $r(g)$'s.
For example, $|r(g)|=\ord_{p^m}(r(g))$ and $\GEN{r(g_1),r(g_2),\dots,r(g_k)}$ denotes the group generated by the $r(g_i)$'s  in $\U_{p^m}$, for $g_1,\dots,g_k \in G$. 
	Then $g\mapsto r(g)$ defines  a group homomorphism $G\rightarrow \U_{p^m}$ with kernel $C_G(G')$ and image contained in the Sylow $p$-subgroup of 
$\U_{p^m}$. 
	In particular 
	$$|gC_G(G')|=\ord_{p^m}(r(g)),$$
	and hence
	\begin{equation} \label{oComoOrden}
	o(g)=\begin{cases} 0, & \text{if } r(g)\equiv -1 \bmod p^m; \\ \log_p(
	\ord_{p^m}(r(g))), & \text{otherwise}.\end{cases}.
	\end{equation} 
Therefore, $|gC_G(G')|=p^e$ with $0\le e \le m-1$ and if $p=2$ and $m\ge 3$ then $e\le m-2$. Furthermore, if $p=2$, $m=2$ and $e=1$ then $r(g)\equiv -1 \bmod 4$ and hence $o(g)=0$.
	This implies that 
	\begin{equation}\label{Cotao}
	0\le o_i <m \qand \text{if }p=2 \text{ and } m\ge 2 \text{ then } o_i\le m-2.
	\end{equation}
If $p$ is odd or $m\le 2$ then $\U_{p^m}$ is cyclic and its subgroup of order $p^e$ for $e\le m-1$ is $\GEN{1+p^{m-e}}$.
If $m\ge 3$ then $\U_{2^m}=\GEN{5}\times \GEN{-1}$ with $\ord_{2^m}(5)=2^{m-2}$ and $\ord_{2^m}(-1)=2$. 
Thus, in this case, $\U_{2^m}$ has exactly three subgroups of order $2$, namely $\GEN{-1}$, $\GEN{1+2^{m-1}}$ and $\GEN{-1+2^{m-1}}$, and exactly two cyclic subgroups of order $p^{e}$ for $e\in \{2,\dots,m-2\}$, namely $\GEN{1+2^{m-e}}$ and $\GEN{-1+2^{m-e}}$.
Hence $\GEN{r(g)}$ is determined by $o(g)$ and $\sigma(g)$, namely:
\begin{equation}\label{rosigma}
\GEN{r(g)}=\GEN{\sigma(g)(1+p^{m-o(g)})}=\GEN{\sigma(g)+p^{m-o(g)}}.
\end{equation}
Moreover, if $g,h\in G$ and $o(g)\le o(h)$ then there exist an integer $x$ such that $p\nmid x$ and $r(gh^{-xp^{o(h)-o(g)}})\equiv \pm 1\bmod p^m$,  with negative sign occurring exactly when $p=2$, $m\ge 3$ and either $\sigma(g)=-1$ and $o(h)>o(g)$, or $o(h)=o(g)$ and $\sigma(g)\ne \sigma(h)$.
We will use this without specific mention. 

Another fact that we will use without specific mention is the following: if $(b_1,b_2) \in \B$ then
$$\B =  \{(b_1^{x_1}b_2^{y_1}[b_2,b_1]^{z_1},b_1^{x_2}b_2^{y_2}[b_2,b_1]^{z_2}) : x_i,y_i,z_i\in \Z, \ 
p^{n_1-n_2} \mid x_2, \text{ and } x_1y_2\not\equiv x_2y_1\bmod p\}.$$

Our first objective is to characterize the elements $b$ of $\B$ for which $(\sigma(b_1),\sigma(b_2),o(b_1),o(b_2))$ achieves the maximum $\sigma o$, i.e. the elements of the following set:
$$\B' = \{b\in \B : \sigma o=(\sigma(b_1),\sigma(b_2),o(b_1),o(b_2))\}.$$

\begin{lemma}\label{Fijando-rLema}
	Let $b=(b_1,b_2)\in \B$. Then $b\in \B'$ if and only if the following conditions hold:
	\begin{enumerate}
		\item\label{sigma11} If $\sigma(b_1)=1$ then $\sigma(b_2)=1$.
		\item\label{nsIguales} If $n_1=n_2$ then $\sigma(b_1)=\sigma(b_2)$.
		\item\label{OrdenOs} One of the following conditions holds:
		\begin{enumerate}
			\item\label{o10}  $o(b_1)=0$. 
			\item\label{osIguales} $0<o(b_1)=o(b_2)$ and $\sigma(b)=(-1,-1)$.		
			\item\label{o20}  $0=o(b_2)<o(b_1)$ and  $n_2<n_1$. 
			\item\label{osNo0}  $0< o(b_2) < o(b_1) < o(b_2) +n_1-n_2$. In particular, $n_2<n_1$.
		\end{enumerate}   
	\end{enumerate}
\end{lemma}

\begin{proof}
Suppose that $b\in \B'$. Most of the arguments of this part of the proof are by contradiction. More precisely, under the assumption that one of the conditions fails we construct another element $(\overline{b}_1,\overline{b}_2)\in \B$ with $(\sigma(b_1),\sigma(b_2),o(b_1),o(b_2))>_{lex} (\sigma(\overline{b}_1),\sigma(\overline{b}_2),o(\overline{b}_1),o(\overline{b}_2))$.
	
	If $\sigma(b_1)=1$ and $\sigma(b_2)=-1$ then $\overline{b}=(b_1b_2,b_2)\in \B$ and $-1=\sigma(\overline{b}_1)<\sigma(b_1)$ contradicting the minimality.
	This proves \eqref{sigma11}.
	
	If $n_1=n_2$ and $\sigma(b_1)\ne \sigma(b_2)$ then $\sigma(b_1)=-1$ and $\sigma(b_2)=1$. Then $\overline{b}=(b_1,b_1b_2)\in \B$ with $\sigma(b_1)=\sigma(\overline{b}_1)$ and $\sigma(\overline{b}_2)=-1<\sigma(b_2)$, contradicting the minimality.
	This proves \eqref{nsIguales}.
	
	Assume first that $o(b_2)\geq o(b_1)$. 
	Then $r(b_1b_2^{-xp^{o(b_2)-o(b_1)}})=\pm 1$ for some integer $x$ with  $p\nmid x$ and negative sign in case $p=2$, $\sigma(b_1)=-1$ and either $\sigma(b_2)=1$ or $o(b_2)>o(b_1)$.  
	Then $\overline b=(b_1 b_2^{-xp^{o(b_2)-o(b_1)}},b_2) \in \B$ and hence $o(\overline{b}_1)=0$. 
	If moreover $\sigma(b_1)=1$ then $\sigma(\overline{b}_1)=\sigma(b_1)$ and hence, since $\sigma(b_2)=\sigma(\bb_2)$, necessarily $o(b_1)=0$. 
	If $\sigma(b_1)=-1$, and either $\sigma(b_2)=1$ or $o(b_2)>o(b_1)$ then also $\sigma(\overline{b}_1)=-1$ so that $o(b_1)=0$. 
	Thus in this case either \eqref{o10} or \eqref{osIguales} holds.  
	
	Now assume that  $o(b_1)>o(b_2)$.
	This implies that $n_2<n_1$, since otherwise both $\overline b=(b_2,b_1)$ and $\widehat{b}=(b_1,b_1b_2)$ belong to $\B$, and we obtain a contradiction because if $\sigma(b_2)=\sigma(b_1)$ then $\sigma(\overline{b}_1)=\sigma(b_2)=\sigma(b_1)$ and $o(\overline{b}_1)=o(b_2)<o(b_1)$, contradicting the minimality, and otherwise, i.e. if $\sigma(b_1)\ne \sigma(b_2)$ then $\sigma(b_1)=\sigma(\widehat{b}_1)$ and $\sigma_2(\widehat{b})=-1<\sigma(b_2)$.
	Thus, if $o(b_2)=0$ then condition \eqref{o20} holds. 
	Assume otherwise, so $o(b_2)\neq 0$ and let $x$ be a integer coprime to $p$ such that $r(b_2b_1^{-xp^{o_1-o_2}})\equiv\pm 1\bmod p^m$.
	If $o(b_2)+n_1-n_2\leq o(b_1)$ then $\overline b=(b_1,b_1^{-xp^{o(b_1)-o(b_2)}}b_2) \in \B$, $\sigma(\bb_1)=\sigma(b_1)$,  $\sigma(\overline{b}_2)=\sigma(b_2)$,  $o(\overline{b}_1)=o(b_1)$, and $o(\overline b_2)=0<o(b_2)$, a contradiction. 
	Thus $o(b_1)<o(b_2)+n_1-n_2$ and condition \eqref{osNo0} holds.

	Conversely, assume that $b$ verifies the conditions (1)-(3) and 
	let $s_i=r(b_i)$ for $i=1,2$. 
	By the minimality $(\sigma_1,\sigma_2,o_1 ,o_2 )\leq_{\lex} (\sigma(b_1),\sigma(b_2),o(b_1),o(b_2))$ and we have to prove that the equality holds.   To this end  fix $\overline b \in \B'$, and take integers $x_i,$ and $y_i$  such that $\overline  b_iG'= b_1^{x_i}  b_2^{y_i}G'$ for $i=1,2$. Thus $o_i=o_i(\overline b)$, $\sigma_i=\sigma_i(\overline b)$ and $r(\overline b_i)\equiv s_1^{x_i}s_2^{y_i} \bmod p^m$.
	
	Of course if $\sigma(b_1)=-1$ then $\sigma(b_1)=\sigma_1$. 
	Otherwise, $\sigma(b_1)=\sigma(b_2)=1$ by condition (1), and hence $\sigma_i(\overline{b})=1$ for $i=1,2$. This proves that $\sigma(b_1)=\sigma_1$.
	
	If $\sigma(b_2)\ne \sigma_2$ then $\sigma(b_1)=-1=\sigma_1$ and $\sigma(b_2)=1$. Then $p=2$, $n_1\ne n_2$ by (2), and hence $x_2$ is even and $y_2$ is odd, which implies that $\sigma(b_2)=\sigma_2$, a contradiction. Thus $\sigma(b_2)=\sigma_2$.

	By means of contradiction suppose that $o_1<o(b_1)$. 
	In particular $o(b_1)\ne 0$, i.e. $b$ does not satisfy (3a). 
	Suppose that condition (3b) holds. Then $o(b_1)=o(b_2)$ and $\sigma_1=\sigma(b_1)=\sigma(b_2)=-1$. 
	Thus $p=2$, $m\ge 2$ and $\GEN{s_1}=\GEN{s_2}$.
	Therefore $s_1\equiv s_2 \bmod 4$ and hence $s_1^{x_1+y_1}\equiv s_1^{x_1}s_2^{y_1}\equiv r(\overline b_1)\equiv -1 \bmod 4$.
	Then $x_1\not\equiv y_1\bmod 2$ and therefore $|\overline b_1C_G(G')|=|b_1C_G(G')|=2^{o(b_1)}\ne 2^{o_1}$.
	Thus, by \eqref{oComoOrden}, $o_1=0$ and $s_1^{x_1}s_2^{y_1}\equiv -1 \bmod 2^m$.
	As $o(b_1)=o(b_2)>o_1$ and $x_1\not\equiv y_1\bmod 2$, it follows that $o(b_1) =o(b_2)=1$, $m\ge 3$ and either $2\nmid x_1$ and $s_1\equiv -1+2^{m-1} \bmod 2^m$ or $2\nmid y_1$ and $s_2\equiv -1+2^{m-1} \bmod 2^m$. In both cases $-1\equiv -1+2^{m-1}\bmod 2^m$, a contradiction.
	This proves that  $o(b_2)<o(b_1)$ and  $n_2<n_1$. 
	Therefore $p \mid x_2$, so $p\nmid x_1y_2$ and 
	$|\overline b_1C_G(G')| =\ord_{p^m}(s_1^{x_1}s_2^{y_1})=\ord_{p^m}(s_1)=p^{o(b_1)}\ne p^{o_1}$. 
	Again this implies that $o_1=0$, $\sigma_1=-1$ and $o(b_1)=1$, so that $o(b_2)=0$ and $-1\equiv s_1^{x_1}s_2^{y_1} \equiv \pm s_1 \bmod 2^m$ which is not possible because $s_1\not\in \GEN{-1}$, as $o(b_1)=1$. This proves that $o_1=o(b_1)$.

	Finally if, $o_2\ne o(b_2)$ then $o(b_2)\ne 0$.
	Hence $b$ does not satisfy (3c). 
	If $o_1=0$ then $\pm 1 \equiv \pm s_2^{y_1} \bmod p^m$ and the signs must agree for otherwise $p=2$, $m\ge 2$ and $-1\equiv s_2^{y_1}\bmod 2^m$ which is  only possible if $y_1$ is odd and $s_2\equiv -1 \bmod 2^m$ in contradiction with $o(b_2)\ne 0$, i.e. $b$ does not satisfy (3a).
	If $o_1=o_2$ then $o(b_1)=o_1=o_2<o(b_2)$ and hence, by assumption, $o_1=0$, which we have just seen that it is not possible. 
	Thus $b$ does not satisfy (3b) either. 
	Hence, (3d) holds, i.e. $0<o(b_2)<o(b_1) <o(b_2)+n_1-n_2$. 
	Thence $p^{o(b_1)-o(b_2)+1}\mid x_2$ and $p\nmid y_2$.
	Therefore $p^{o_2}<p^{o(b_2)}=\ord_{p^m}(s_2)= \ord_{p^m}(s_1^{x_2} 
s_2^{y_2})=\ord_{p^m}(r(\bb_2))$. Then, by \eqref{oComoOrden}, $p=2$, $o_2=0$, $o(b_2)=1$ and $\sigma_2=-1$, yielding the following  contradiction: $-1 \equiv r(\bb_2) \equiv s_1^{x_2} s_2^{y_2} \equiv-1 +2^{m-1} \bmod 2^m$. Hence $o_2=o(b_2)$.
\end{proof}

\begin{proposition}\label{Fijando-r}
	Let $p$ be a prime  integer and let $G$ be a non-abelian group with $G'\cong C_{p^m}$ and  $G/G'\cong C_{p^{n_1}}\times C_{p^{n_2}}$ with $n_2\le n_1$.
	Let $\sigma o=(\sigma_1,\sigma_2,o_1,o_2)$ and let $r_1$ and $r_2$ be given as in \eqref{def-erres}.
	\begin{enumerate}
		\item\label{ImparSigma1} If $p\ne 2$ then $\sigma_1=1$.
		\item\label{sigma1} If $\sigma_1=1$ then $\sigma_2=1$. 
		\item\label{nIgualSigmaIgual} If $n_1=n_2$ then $\sigma_1=\sigma_2$.		
		\item\label{CondicionesOs} One of the following conditions holds:
		\begin{enumerate}
			\item\label{O10}  $o_1=0$. 
			\item\label{OIgual} $0<o_1=o_2$ and $\sigma_1=\sigma_2=-1$.		
			\item\label{O20}  $0=o_2<o_1$ and  $n_2<n_1$. 
			\item\label{OsNo0}  $0< o_2 < o_1 < o_2 +n_1-n_2$. In particular, $n_2<n_1$.
		\end{enumerate}   
		\item\label{bsConErres} $\B_r\ne \emptyset$, i.e. $\B$ contains an element $(b_1,b_2)$ such that $a^{b_i}=a^{r_i}$ for every $a\in G'$ and $i=1,2$.
	\end{enumerate}
\end{proposition}

\begin{proof}
\eqref{ImparSigma1} is a direct consequence of \eqref{sigmag=-1}.
Statements \eqref{sigma1}, \eqref{nIgualSigmaIgual} and \eqref{CondicionesOs} follow directly from \Cref{Fijando-rLema}.
Fix $(b_1,b_2)\in \B'$.
	Using \eqref{rosigma} it easily follows that $r_i$ and $r(b_i)$ generate  	the same multiplicative group in $\U_{p^m}$.
	Thus there are integers $x$ and $y$ with $p\nmid xy$ and $r_i=r(b_i)^{x_i}$. 
	Then $(\bb_1,\bb_2)=(b_1^x,b_2^y)\in \B$  and $r(\overline{b}_i)=r_i$, i.e. $a^{b_i}=a^{r_i}$ for every $a\in G'$.
	Therefore $(\bb_1,\bb_2)\in \B_r$.
\end{proof}

In \Cref{Fijando-r} we have obtained some restrictions for $\sigma o$. 
We now obtain some restrictions on the $o'_i$'s and $u_i$'s. 
To this end, we fix $b=(b_1,b_2)\in \B_r$.
Recall that $|b_i|=p^{n_i+o'_i(b)}$ and hence 
\begin{equation}\label{to'u}
0\le o'_i(b)=m-v_p(t_i(b))\le m,  \quad 1\le u_i(b) \le p^{o'_i(b)} \qand p\nmid u_i(b) \quad (i=1,2).
\end{equation}
From \eqref{Def-TesFunciones} and \eqref{CBAPotenciaga} it follows
\begin{eqnarray}
\label{RelprN}   r_i^{p^{n_i}} & \equiv & 1 \bmod p^m, \\
\label{Relptr}   t_i(b)r_i & \equiv & t_i(b) \bmod p^m \\
\label{Relpt1r2} \Ese{r_1}{p^{n_1}} & \equiv & t_1(b)(1-r_2) \bmod p^m, \\
\label{Relpt2r1} \Ese{r_2}{p^{n_2}}& \equiv & t_2(b)(r_1-1)\bmod p^m.
\end{eqnarray}

\begin{lemma}\label{oo'}
	The following statements hold for every $b\in \B_r$:
	\begin{enumerate}
		\item\label{Cotao'}   $o'_i(b)\le m-o_i$ and if $\sigma_i=-1$ then $o'_i(b)\le 1$ and $u_i(b)=1$, for $i=1,2$.
		\item\label{oMenorn} $o_i<n_i$, for $i=1,2$.
		\item\label{oo'Sigmas1} If $\sigma_1=1$ then the following conditions 
hold:
		\begin{enumerate}
			\item $o_2+o'_1(b)\le m\le n_1$ and if $m=n_1$ then $o_1o_2=0$.
			\item Either $o_1+o'_2(b)\le m\le n_2$ or $2m-o_1-o'_2(b)=n_2<m$ and 
$u_2(b)\equiv 1 \bmod p^{m-n_2}$.
		\end{enumerate}
		\item\label{oo'SigmasDistintas} If $\sigma_1\ne \sigma_2$ then one of the following conditions hold: 
		\begin{enumerate}
			\item $m\le n_2$ and $o'_2(b)\le 1$.
			\item $m-o'_2(b)+1=n_2<m$, $u_2(b)(1+2^{m-o_1-1})\equiv - 1 \bmod 2^{m-n_2}$ and  $1\le u_2(b) \le 2^{m-n_2+1}$.
		\end{enumerate}
	\end{enumerate}
\end{lemma}

\begin{proof}
	\eqref{Cotao'}  Suppose firstly that $\sigma_i=-1$. 
	Then $p=2$, $m\ge 2$ and $r_i\equiv -1 \bmod 4$, by \eqref{sigmag=-1}. Then $v_2(r_i-1)=2$ and from \eqref{to'u} and \eqref{Relptr} it follows that $m-1\le v_2(t_i(b)) = m-o'_i(b)$, so that $o'_i(b)\le 1$ and $o_i+o'_i(b)\le m-1$, by \eqref{Cotao}.
	The former implies that $u_i(b)=1$.
	Suppose otherwise that $\sigma_i=1$. Then $v_p(r_i-1)=m-o_i$ and \eqref{to'u} and \eqref{Relptr} imply that $m\le v_p(t_i(b))+v_p(r_i-1)=2m-(o_i+o'_i(b))$.
	
	For the remainder of the proof we consider separately the different values of $\sigma_1$ and $\sigma_2$.
	
	Suppose that $\sigma_1=1$. 
	Then $\sigma_2=1$, by \Cref{Fijando-r}.\eqref{sigma1}. 
	Hence $v_p(r_i-1)=m-o_i$ and either $p$ is odd or $r_1\equiv r_2 \equiv
1 \bmod 4$.
	Thus $\ord_{p^m}(r_i)=p^{o_i}$ and $v_p(\Ese{r_i}{p^{n_i}})=n_i$, by 
\Cref{EseProp}.\eqref{ValEse}. 
	Then $o_i\le n_i$, by \eqref{RelprN} and combining this with \eqref{to'u}, \eqref{Relpt1r2} and \eqref{Relpt2r1} we deduce that
	$$o_2+o'_1(b)\le m \le n_1 \text{ or } n_1=2m-o_2-o'_1(b)<m$$
	and
	$$o_1+o'_2(b)\le m \le n_2 \text{ or } n_2=2m-o_1-o'_2(b)<m.$$
As $n_2\le n_1$, if $n_1<m$ then we obtain a contradiction because
	$$2m > n_1+n_2 = 4m-( o_2+o'_1(b)+o_1+o'_2(b)) \ge 2m,$$
by \eqref{Cotao'}.
Thus $o_2+o'_1(b)\le m\le n_1$ and hence $o_1<m\le n_1$.
If $o_2\ge n_2$ then $n_2<m$ and hence
	$$m\ge  o_2+o'_2(b)\geq n_2+o'_2(b)=2m-o_1>m,$$ a contradiction.
This proves \eqref{oMenorn} in this case.
Moreover, if $n_2<m$ then by \eqref{Relpt2r1} and \Cref{PropEse}.\eqref{EseCero}
	$$u_2(b) p^{n_2}\equiv u_2(b) p^{2m-o'_2(b)-o_1}\equiv  t_2(b) (r_1-1)\equiv  \Ese{r_2}{p^{n_2}} \equiv p^{n_2} \bmod 2^m,$$
and hence $u_2(b)\equiv 1 \bmod p^{m-n_2}$.
Finally assume that $m=n_1$. If $n_2=n_1$ then $o_1=0$, by \Cref{Fijando-r}.\eqref{CondicionesOs}.
Otherwise $n_2<n_1=m$ and hence $n_2=2m-o_1-o'_2(b)$.
Then, by \eqref{Cotao'},
	$$n_1-n_2=m-n_2=-m+o_1+o_2'(b)\le o_1-o_2.$$
Therefore $o_2=0$, by \Cref{Fijando-r}.\eqref{CondicionesOs}.
	This finishes the proof of \eqref{oo'Sigmas1}.
	
	Suppose that $\sigma_2=-1$. Then $\sigma_1=-1$, by \Cref{Fijando-r}.\eqref{sigma1} and hence $p=2$, $m\ge 2$,  $r_i\equiv -1 \bmod 4$ and $v_2(r_i+1)=m-o_i$ for $i=1,2$.
	Moreover, $v_p(t_i(b))=m-o'_i(b)\ge m-1$, 
	by \eqref{Cotao'} and therefore $\Ese{r_i}{2^{n_i}}\equiv 0 \bmod 2^m$ by
\eqref{Relpt1r2} and \eqref{Relpt2r1}. 
By \Cref{EseProp}.\eqref{ValEse-1}
	$$m\le v_2(\Ese{r_i}{2^{n_i}}) = v_2(r_i+1)+n_i-1=m-o_i+n_i-1,$$
that is $o_i\le n_i-1$, proving \eqref{oMenorn} in this case.
	
Finally, suppose that $\sigma_1=-1$ and $\sigma_2=1$. Then $p=2$, $m\ge 2$, $r_1\equiv -1 \bmod 4$ and $r_2\equiv 1 \bmod 4$. By \eqref{Cotao'},
$o'_1(b)\le 1$ and hence $v_2(t_1(b))=m-o'_1(b)\ge m-1$. Thus $t_1(b)(r_2-1)\equiv 0 \bmod 2^m$ and, by \eqref{Relpt1r2} and \Cref{EseProp}.\eqref{ValEse-1},
	$$m\le v_2(\Ese{r_1}{2^{n_1}})=n_1+v_2(r_1+1)-1=n_1+m-o_1-1,$$
i.e. $o_1\le n_1-1$.
Moreover, as $r_2\equiv 1 \bmod 4$, by \Cref{EseProp}.\eqref{ValEse}, $v_2(\Ese{r_2}{2^{n_2}})=n_2$ and hence \eqref{Relpt2r1} implies that either $m\le n_2$ and $m\le v_p(t_2(b))+v_2(r_1-1)=m-o'_2(b)+1$ or $m>n_2=m-o'_2(b)+1$.
In the former case $o_2<m\le n_2$ and $o'_2(b)\le 1$.
In the latter case, by \eqref{Relpt2r1} and \Cref{EseProp}.\eqref{EseCero},
	$$u_2(b)2^{n_2}(-1 - 2^{m-o_1-1})=u_2(b)2^{n_2-1}(-2  - 2^{m-o_1})=t_2(b)(r_1-1)\equiv \Ese{r_2}{p^{n_2}} \equiv 2^{n_2} \bmod 2^m.$$
Thus $o_2\le m-o'_2(b)<n_2$, which finishes the proof of \eqref{oMenorn}, and $u_2(b)(-1 - 2^{m-o_1})\equiv 1 \bmod 2^{m-n_2}$, which finishes the proof of \eqref{oo'SigmasDistintas}.
\end{proof}

Combining \eqref{Cotao} and \Cref{oo'} we obtain the following:

\begin{corollary}\label{Cotasoo'}
	Let $G$ be a non-abelian cyclic-by-abelian finite group of prime-power order with
	$\inv(G)=(p,m,n_1,n_2,\sigma_1,\sigma_2,o_1,o_2,o'_1,o'_2,u_1,u_2)$. 
	Then the following statements hold:
	\begin{enumerate}
		\item\label{Cotao'G} $o'_i\le m-o_i$, $o_1'\leq m-o_2$ and if $\sigma_i=-1$ then $o'_i\le 1$ and $u_i=1$.
		\item\label{CotaOs} $o_i\le \min(m,n_i)-1$ and if $p=2$ and $m\ge 2$ then $o_i\le m-2$.					
		\item\label{m>n1} If $m> n_1$ then $\sigma_1=-1$.
		\item\label{m=n1} If $\sigma_1=1$ and $m=n_1$ then $o_1o_2=0$.
		\item  Suppose that $\sigma_1\ne \sigma_2$.
		\begin{enumerate}	
		\item\label{sigmasDistintosm>=} If $m\le n_2$ then $o'_2\le 1$.
		\item\label{sigmasDistintosm<} If $m>n_2$ then $o'_2=m+1-n_2$, $u_2(1+2^{m-o_1-1}) \equiv -1 \bmod  2^{m-n_2}$ and $1\le u_2 \le 2^{m-n_2+1}$.
		\end{enumerate} 
	\end{enumerate}
\end{corollary}

\section{Changing bases within $\B_r$}\label{SectionChangingBases}

In the previous section we proved that $\B_r\ne \emptyset$.
All throughout this section $b=(b_1,b_2)\in \B_r$ and $\overline b=(\overline{b}_1,\overline{b}_2)$ with $\overline b_i\equiv b_1^{x_i} b_2^{y_i} [b_2,b_1]^{z_i}$, where $x_i,y_i$ and $z_i$ are integers, for $i=1,2$.

Next lemma characterizes when $\bb$ belongs to $\B_r$.

\begin{lemma}\label{LemmaEnB2}  
	$\overline b\in \B_r$ if and only if the following conditions hold:
	\begin{enumerate}
		\item \label{EnB0EnB2} $p^{n_1-n_2}\mid x_2$ and $x_1y_2\not \equiv x_2y_1\bmod p$.
		\item\label{B2Sigma} 
	If $\sigma_2=-1$ then $x_1\not\equiv y_1 \bmod 2$ and $x_2\not\equiv y_2 \bmod 2$.
		\item \label{B2os}
		\begin{enumerate}
			\item If $o_1=0$ then $y_1\equiv y_2-1\equiv 0 \bmod p^{o_2}$.
			\item If $0<o_1=o_2$ then $x_1+y_1\equiv x_2+y_2\equiv 1 \bmod p^{o_1}$.
			\item If $o_2=0<o_1$ then $x_1-1\equiv x_2\equiv 0\bmod p^{o_1}$.
			\item If $0<o_2<o_1$ then $x_1+y_1p^{o_1-o_2}\equiv 1\bmod p^{o_1}$ and $x_2p^{o_2-o_1}+y_2\equiv 1\bmod p^{o_2} $. (Observe that  $x_2p^{o_2-o_1}\in\Z$ because $o_1-o_2< n_1-n_2$, by \Cref{Fijando-r}.\eqref{CondicionesOs}.)
		\end{enumerate}
	\end{enumerate}
\end{lemma}

\begin{proof}  
	Condition (1) only means that $\overline b\in \B$, so we can take it for granted. Under such assumption, $\overline b \in \B_r$ if and only if
	$r_i\equiv r_1^{x_i}r_2^{y_i} \bmod p^m$,  for $i=1,2$.
	Write $R_i=\sigma_i r_i$. 
	Observe that if $\sigma_i=-1$ then $p=2$, $m\ge 2$ and $R_i\equiv 1 \bmod 4$.
	Therefore $\bb\in\B_r$ if and only if 
	\begin{equation}
	\label{sigmaR}
	\sigma_1^{x_1-1} \sigma_2^{y_1}= \sigma_1^{x_2} \sigma_2^{y_2-1}=1 \qand 
	R_1^{x_1-1} R_2^{y_1} \equiv R_1^{x_2} R_2^{y_2-1} \equiv 1 \bmod p^m.
	\end{equation}
By statements \eqref{sigma1} and \eqref{nIgualSigmaIgual} of \Cref{Fijando-r}, the equalities on the left side of \eqref{sigmaR} are equivalent to \eqref{B2Sigma}.
	Moreover, $\ord_{p^m}(R_i)=p^{o_i}$ and if $o_1o_2\ne 0$ then $o_2\le o_1$ and $R_2\equiv R_1^{p^{o_1-o_2}} \bmod p^m$, by \Cref{Fijando-r}.\eqref{CondicionesOs} and \eqref{def-erres}. Using this it follows easily that
the congruences on the right side of \eqref{sigmaR} are equivalent to 
	the conditions in \eqref{B2os}.
\end{proof}

In the remainder of the section we assume that $\bb\in \B_r$. The aim now is to describe the connection between the $o'_i(b)$'s and $u_i(b)$'s with the $o'_i(\bb)$'s and $u_i(\bb)$'s.
We first obtain a general description and then we specialize to the three cases depending on the values of $\sigma_1$ and $\sigma_2$.

\begin{lemma}\label{PreCongruencias}
	If $\bb\in \B_r$ then 
	\begin{eqnarray}
	\label{pret1} {t}_1(\overline b) \alpha &\equiv & x_1t_1(b)+y_1p^{n_1-n_2}t_2(b)+ \beta_1 \bmod p^m, \\
	\label{pret2} {t}_2(\overline b) \alpha  &\equiv & x_2t_1(b)p^{n_2-n_1}+y_2t_2(b)+ \beta_2 \bmod p^m.
	\end{eqnarray} 
	where\begin{eqnarray}
	\alpha&=&\Ese{r_1}{x_1}\Ese{r_2}{y_2}r_2^{y_1}-\Ese{r_1}{x_2}\Ese{r_2}{y_1}r_2^{y_2}, \\ 
	\beta_i &=& 
	\Ese{r_1}{x_i} \Ese{r_2}{y_i}\Ese{r_1^{x_i},r_2^{y_i}}{p^{n_i}}  \quad (i=1,2).
	\end{eqnarray}
\end{lemma}
\begin{proof} 
	Let $a=[b_2,b_1]$ and $\overline{a}=[\overline{b}_1,\overline{b}_2]$.
	Observe that the hypothesis implies that $r(b_i)\equiv r(\overline b_i)\equiv r_i \equiv r_1^{x_i}r_2^{y_i} \bmod p^m$.
	Therefore, by \eqref{CBAConmutadores}, for every $v,w\in \Z$
	\begin{equation}\label{ConmutadorabiPotencias}
	[a^v,b_i^w]=a^{v(r_i^w-1)} 
	\end{equation}
	and 
	\begin{equation}\label{Conmutadorb2b1Potencias}
	[b_2^w,b_1^v]=a^{\Ese{r_1}{v}\Ese{r_2}{w}}.
	\end{equation}
	 By \eqref{CBAConmutadores}, \eqref{ConmutadorabiPotencias} and \eqref{Conmutadorb2b1Potencias}
	$$\overline{a}=a^{\alpha+\gamma},$$
	where $\gamma=z_2(r_1-1) + z_1(1-r_2)$;
	by \eqref{CBAPotenciaga} and \eqref{CBAPotenciagh}
	\begin{equation} \label{ErrorDeMiech}
	\overline{b}_i^{p^{n_i}}=b_1^{x_ip^{n_i}}b_2^{y_ip^{n_i}}a^{ \beta_i +z_i\Ese{r_i}{p^{n_i}}}
	\end{equation} 
	and as $p^{n_1-n_2}\mid x_2$
	\begin{eqnarray*}
		a^{t_1(\bb)(\alpha+\gamma)} = \overline{a}^{t_1(\bb)}=\overline{b}_1^{p^{n_1}}=a^{x_1t_1(b)+y_1p^{n_1-n_2}t_2(b)+ \beta_1+z_1\Ese{r_1}{p^{n_1}}} & \qand  \\
		a^{t_2(\bb)(\alpha+\gamma)} =\overline{a}^{t_2(\bb)}=\overline{b}_2^{p^{n_2}}=
		a^{x_2t_1(b)p^{n_2-n_1}+y_2t_2(b)+ \beta_2+z_2\Ese{r_2}{p^{n_2}}} .&
	\end{eqnarray*} 
	By \eqref{Relptr}, \eqref{Relpt1r2} and \eqref{Relpt2r1} it is clear that $\gamma t_i(\bb) \equiv z_i \Ese{r_i}{p^{n_i}} \bmod p^m$ for $i=1,2$, so the statement follows.
\end{proof}

\begin{remark}\label{ErrorMiech} 
	{\rm An analogue of \Cref{PreCongruencias} was stated within the proof of 
		\cite[Theorem~9]{Miech1975}. 
		Unfortunately, there is a mistake which leads to an incorrect version of condition \eqref{pret2}. Namely, transferring to our notation, the exponent of $a$ in \eqref{ErrorDeMiech} for $i=2$ appears as $\beta_2+ z_2(r_2-1) \Te{r_1,r_2}{p^{n_2}}$ in \cite{Miech1975}, rather than $\beta_2 +z_2 \Ese{r_2}{p^{n_2}}$. As a consequence some groups are missing in \cite{Miech1975}. A minimal example of such situation is formed by the three groups $G$ with $\inv(G)=(3,3,3,2,1,1,2,0,1,2,1,u)$ for  $u\in \{1, 4, 7\}$. Only the group with  $u=1$ appears in \cite{Miech1975}.
}\end{remark}

\begin{lemma} \label{tt++}
	If $\bb\in \B_r$ and $\sigma_1=1$ then 
	\begin{equation}
	\label{t1++}   (x_1y_2-x_2y_1) u_ 1(\bb) p^{m-o_1'(\bb)} \equiv  x_1u_1(b) p^{m-o_1'(b)}+y_1u_2(b)p^{m-o_2'(b)+n_1-n_2}+B_1 \bmod p^m
	\end{equation}
and
	\begin{equation}
	\label{t2++}A+   (x_1y_2-x_2y_1)  u_2(\bb) p^{m-o_2'(\bb)} \equiv  x_2u_1(b)p^{m-o_1'(b)+n_2-n_1}+y_2u_2(b) p^{m-o_2'(b)}+B_2 \bmod p^m,
	\end{equation}
	where  
$$A=\begin{cases}
((x_1-1) y_2-x_2y_1)2^{n_2-1} , &\text{if } p=2 \text{ and } 0<m-n_2=o_1 -o_2; \\
0 ,& \text{otherwise}; \end{cases}$$
and for $i=1,2$
	$$B_i=\begin{cases}
	x_iy_i 2^{n_i-1},& \text{if }p=2 \text{ and }m=n_1;\\
	0, &\text{otherwise}.
	\end{cases}$$
\end{lemma}

\begin{proof}
	By \Cref{PreCongruencias}, it is enough to prove the following
	\begin{eqnarray} 
		t_1(\overline b)\alpha &\equiv& t_1(\overline b)(x_1y_2-x_2y_1) \bmod p^m; \label{t1}\\
		t_2 (\overline b)\alpha&\equiv& A + t_2 (\overline b)\left(  x_1y_2-x_2y_1\right)  \bmod p^m, \label{t2} \\
		\beta_i     &\equiv & B_i \bmod p^m \quad (i=1,2).   \label{beta}
	\end{eqnarray}	
	
	By \Cref{Cotasoo'}.\eqref{m>n1} and \Cref{EseProp}.\eqref{ValEse} $v_p(\Ese{r_1}{p^{n_1}})=n_1\geq m$, so	using \eqref{Relptr} and \eqref{Relpt1r2},   it follows that $ {t}_1(\overline b)\alpha\equiv  {t}_1(\overline b)(x_1y_2-x_2y_1)\bmod p^m$ and
	$t_2(\overline b)\alpha \equiv  {t}_2(\overline b)(\Ese{r_1}{x_1}y_2-  \Ese{r_1}{x_2} y_1)\bmod p^m$.
	This proves \eqref{t1} and reduces the proof of \eqref{t2} to prove the following:
	\begin{align*}
	t_2(\bb)\Ese{r_1}{x_1 }\equiv &\begin{cases}
	t_2(\bb)x_1+(x_1-1)2^{n_2 -1}  \bmod 2^{m},  & \text{if }p=2  \text{ and } 0<m -n_2=o_1-o_2 ; \\
	t_2(\bb)x_1 \bmod p^{m }, &\text{otherwise};
	\end{cases}.\\
	t_2(\bb)\Ese{r_1}{x_2}\equiv & \begin{cases}
	t_2(\bb)x_2+x_22^{n_2 -1}  \bmod 2^{m },  & \text{if }p=2  \text{ and }
	0 < m -n_2=o_1-o_2; \\
	t_2(\bb)x_2 \bmod p^{m },&\text{otherwise}.
	\end{cases}
	\end{align*}
	By \Cref{oo'}.\eqref{oo'Sigmas1}, if $m\le n_2$ then $v_p(t_2(\bb))+v_p(r_1-1)=2m-o_1-o'_2(\bb)\ge m$, and hence $t_2(\bb)r_i\equiv r_i\bmod p^m$ and $t_2(\bb)\Ese{r_i}{x_i}\equiv t_2(\bb)x_i \bmod p^m$.
	Otherwise $2m-o_1-o'_2(\bb)=n_2<m$, and since $o'_2(\bb)\le m$, it follows that $o_1\ne 0$.
	Then, by \Cref{oo'}.\eqref{Cotao'} and \Cref{LemmaEnB2}.\eqref{B2os}, $m-n_2=o'_2(\bb)+o_1-m \le o_1-o_2\le \min(v_p(x_1-1),v_p(x_2))$.
	Applying \Cref{EseProp}.\eqref{EseCero} we derive that, for $i=1,2$,
	$$\Ese{r_1}{x_i}\equiv \begin{cases}
	x_i+2^{o'_2(\bb)-1}  \bmod 2^{o'_2(\bb)},  & \text{if } p=2, m-n_2=o_1-o_2 \text{ and } x_i\not\equiv 1 \bmod 2^{m-n_2+1}; \\
	x_i \bmod p^{o'_2(\bb)}, &\text{otherwise}.
	\end{cases}$$
	As $o'_2(\bb)=m-v_p(t_2(\bb))$,
	$$t_2(\bb)\Ese{r_1}{x_i}\equiv \begin{cases}
	t_2(\bb)x_i+2^{m-1}  \bmod 2^m,  & \text{if } p=2, m-n_2=o_1-o_2
	\text{ and }x_i\not\equiv 0 \bmod 2^{m-n_2+1}; \\
	t_2(\bb)x_i \bmod p^m, & \text{otherwise}.
	\end{cases}$$
	Having once again in mind that $0<m-n_2\le \min(v_p(x_2),v_p(x_1-1))$ we deduce that in this case
	\begin{align*}
	t_2(\bb)\Ese{r_1}{x_2}\equiv & \begin{cases}
	t_2(\bb)x_2+x_22^{n_2 -1}  \bmod 2^{m},  & \text{if }p=2  \text{ and }m -n_2=o_1-o_2; \\
	t_2(\bb)x_2 \bmod p^{m }&\text{otherwise};
	\end{cases} .\\ 
	t_2(\bb)\Ese{r_1}{x_1 }\equiv &\begin{cases}
	t_2(\bb)x_1+(x_1-1)2^{n_2 -1}  \bmod 2^{m},  & \text{if }p=2  \text{ and }m -n_2=o_1-o_2 ; \\
	t_2(\bb)x_1 \bmod p^{m }&\text{otherwise};
	\end{cases} .
	\end{align*}
	and \eqref{t2} follows.
	
	By  Lemma \ref{ValEse2}
	$$\beta_1\equiv \begin{cases}
	\Ese{r_1}{x_1}\Ese{r_2}{y_1} 2^{n_1-1} \bmod 2^m, & \text{if } p=2; \\
	0 \bmod p^m, & \text{othewise}.
	\end{cases}.$$
	This proves \eqref{beta} for $i=1$ in case $p\ne 2$ or $m\ne n_1$. Otherwise, by \Cref{PropEse}.\eqref{ValEse},  $\beta_1\equiv x_1y_12^{n_1-1} = B_i \bmod 2^m$, as desired.
	Finally, for the proof of \eqref{beta} for $i=2$, recall that $p^{n_1-n_2} \mid x_2$, and hence, by \Cref{EseProp}.\eqref{ValEse}, $p^{n_1-n_2}\mid \Ese{r_1}{x_2}$.
	Moreover, by Lemma \ref{ValEse2}, if $p>2$ then $p^{n_2} \mid \Te{r_1^{x_2},r_2^{y_2}}{p^{n_2}}$, so $\beta_2\equiv 0 \equiv B_2 \bmod p^m$.
	Assume $p=2$. Then by the same lemma  $ \Te{r_1^{x_2},r_2^{y_2}}{2^{n_2}} \equiv 2^{n_2-1}\bmod2^{n_2} $. Thus $ \Ese{r_1}{x_2} \Te{r_1^{x_2}, r_2^{y_2}}{p^{n_2}} \equiv \Ese{r_1}{x_2} 2^{n_2-1} \bmod 2^{n_1}$. As $v_2(x_2)\ge n_1-n_2$, if $n_1>m$ it is clear that this expression vanishes modulo $2^m$ and hence $\beta_2\equiv B_2 \bmod 2^m$.
	Otherwise, $m=n_1$ and $\Ese{r_2}{x_2}2^{n_2-1}\equiv x_22^{n_2-1} \bmod 2^m$. Thus $\beta_2\equiv x_2y_22^{n_2-1}=B_2 \bmod 2^m$. This proves \eqref{beta}.
\end{proof}

\begin{lemma}\label{ttbb-+}
	If $\bb\in \B_r$, $\sigma_1=-1$ and $\sigma_2=1$  then 

	\begin{eqnarray*} 
		t_1(\overline b)\alpha &\equiv& t_1(\overline b)(x_1y_2-x_2y_1) \bmod 2^m;\\
		t_2(\overline b) \alpha& \equiv& \begin{cases}
			t_2(\overline b)(x_1y_2-x_2y_1) \bmod 2^m ,&\text{if }  n_2 \geq m; \\
			t_2(\overline b) y_2 + ((x_1-1) y_2 -x_2y_1) 2^{m-o_1+o_2-1} \bmod 2^m,
					& \text{if }o_2+1=n_2<m \text{ and }o_1>0 ; \\
			t_2(\overline b ) y_2 \bmod 2^m, &\text{otherwise.}
		\end{cases} \\
		\beta_1 &\equiv &  y_1 2^{n_1-1}\bmod 2^m ;  \\
		\beta_2& \equiv& \begin{cases}
			x_2 2^{m-n_1} \bmod 2^m, & \text{if } n_1=o_1+1,  \ o_2=0<o_1, \text{ and } n_2=1;  \\
			0\bmod 2^m,&\text{otherwise.}
		\end{cases}   
	\end{eqnarray*}	
\end{lemma}

\begin{proof}
Since $\sigma_1=-1$ and $\sigma_2=1$, it follows that $p=2$, $n_1>n_2$, $o'_1(\bb)\le
1$ and $u_1(\bb)=1$ by \Cref{Fijando-r} and \Cref{oo'}.\eqref{Cotao'}.
Moreover, $2^{n_1-n_2}\mid x_2$ and $2\nmid x_1y_2$ by \Cref{LemmaEnB2}, and $t_i(\bb)r_i \equiv  t_i(\bb)\bmod 2^m$, by \eqref{Relptr}.
By \eqref{Relpt1r2}, $t_1(\bb)(r_2-1)\equiv \Ese{r_1}{2^{n_1}} \bmod 2^m$ and, by \Cref{EseProp}.\eqref{ValEse-1} and \Cref{oo'}.\eqref{CotaOs}, $v_2(\Ese{r_1}{2^{n_1}})=n_1+m-o_1-1\ge m$.
So the first congruence follows.

Moreover $t_2(\overline b) \alpha \equiv t_2(\overline b) ( \Ese{r_1}{x_1} y_2-\Ese{r_1}{x_2} y_1) \bmod 2^m$. If $n_2\geq m$ then, by \eqref{Relpt2r1} and \Cref{EseProp}.\eqref{ValEse},
$t_2(\overline b) (r_1-1) \equiv  \Ese{r_2}{2^{n_2}} \equiv 0 \bmod 2^m$, 	so $t_2(\overline b) \alpha \equiv t_2(\overline b) ( x_1 y_2-x_2 y_1) \bmod 2^m$.
Otherwise, $m>n_2=m-o_2'(\overline b) +1$, by \Cref{oo'}.\eqref{oo'SigmasDistintas}.
If $o_1=0$ then $r_1\equiv -1 \bmod 2^m$, so   $\Ese{r_1}{x_1}\equiv 1 \bmod 2^m$ and $\Ese{r_1}{x_2} \equiv  0 \bmod 2^m$, since $2\nmid x_1$ and $2\mid x_2$.
Hence $t_2(\overline b)\alpha \equiv t_2(\overline b)  y_2 \bmod  2^m$.
Assume that $o_1\neq 0 $.
Then $o_1>o_2$ by \Cref{Fijando-r}.\eqref{CondicionesOs}, and $x_1-1\equiv
x_2\equiv 0 \bmod 2^{o_1-o_2 }$ by \Cref{LemmaEnB2}.
So
	$$v_2(t_2(\overline b)\Ese{r_1}{x_2})= m-o_2'(\overline b) + v_2(r_1+1) -1 +v_2(x_2)= 2m-o_2'(\overline b)-o_1-1+v_2(x_2) \geq 2m-o_2'(\overline b)-o_2-1 \geq m -1,$$
by \Cref{oo'}.\eqref{Cotao'}.
Moreover, $v_2(t_2(\overline b)\Ese{r_1}{x_2})=m-1$ if and only if $v_2(x_2)=o_1-o_2 $ and $m-o_2= o_2'(\overline b)$ if and only if $v_2(x_2)=o_1-o_2$ and $n_2=o_2+1$ (because $m-o_2'(\overline b) +1 =n_2$).
Therefore, if $o_1\neq 0$ then
	$$t_2(\overline b) \Ese{r_1}{x_2} \equiv \begin{cases} 
	x_2 2^{ m+o_2-o_1-1} \bmod 2^m , & \text{if }
	o_2+1=n_2<m ;  \\
	0, &\text{otherwise}.
	\end{cases} $$ 
Similar arguments yield
	$$t_2(\overline b) \Ese{r_1}{x_1}=  t_2(\overline b)  + t_2(\overline b) r_1\Ese{r_1}{x_1-1} \equiv  \begin{cases} 
	t_2(\overline b) +(x_1-1) 2^{ m+o_2-o_1-1} \bmod 2^m , & \text{if } o_2+1=n_2<m;  \\
	t_2(\overline b) \bmod 2^m , &\text{otherwise.}
	\end{cases} $$
This proves the second congruence of the lemma.
	
As $x_1-1$ is even, by \Cref{EseProp}.\eqref{ValEse-1}, \Cref{ValEse2} and \Cref{oo'}.\eqref{CotaOs}
		$$v_2(\Ese{r_1}{x_1-1}\Te{r_1^{x_1}, r_2^{y_1}}{2^{n_1}})=v_2(x_1-1) +m-o_1-1 +n_1-1\ge m-o_1-1+n_1\geq m.$$
By \Cref{PropEse}, $y_1+(r_2-1)\Te{r_2,1}{y_1} = \Ese{r_2}{y_1}$, and by \Cref{ValEse2},  	$$v_2((r_2-1) \Te{r_1^{x_1},r_2^{y_1}}{2^{n_1}})=m-o_2 +n_1-1\geq m-o_2+n_2-1\geq m.$$
Thus, using \Cref{PropEse}.\eqref{EseCon1} and \Cref{ValEse2Rebuscado} we conclude that
	\begin{align*}
	\beta_1&= \Ese{r_1}{x_1}\Ese{r_2}{y_1} \Te{r_1^{x_1}, r_2^{y_1}}{2^{n_1}}   =  (1+ r_1\Ese{r_1}{x_1-1})\Ese{r_2}{y_1} \Te{r_1^{x_1}, r_2^{y_1}}{2^{n_1}}\\
	&\equiv \Ese{r_2}{y_1} \Te{r_1^{x_1}, r_2^{y_1}}{2^{n_1}}
	= (y_1 +(r_2-1) \Te{r_2,1}{y_1})  \Te{r_1^{x_1}, r_2^{y_1}}{2^{n_1}}
	\equiv y_1 \Te{r_1^{x_1}, r_2^{y_1}}{2^{n_1}} \\
	&\equiv y_12^{n_1-1} \bmod 2^m.
	\end{align*}
	This proves the third congruence. 
	
\Cref{PropEse} and \Cref{oo'}.\eqref{CotaOs} yield
	$$v_2( \beta_2)=v_2(\Ese{r_1}{x_2}\Ese{r_2}{y_2} \Te{r_1^{x_2}, r_2^{y_2}}{2^{n_2}})=  v_2(r_1+1) -1 +v_2(x_2)   +n_2-1 \geq m-o_1 -1+n_1-1 \geq m-1.$$
Therefore $v_2(\beta_2)=m-1$ if and only if $ n_1=o_1+1$ and $v_2(x_2)=n_1-n_2$.
In that case $o_1>0$, since $n_1>n_2$; and $o_2+n_1-n_2=o_2+o_1-n_2+1\le o_1$, by \Cref{oo'}.\eqref{CotaOs}.
Using \Cref{Fijando-r}.\eqref{CondicionesOs} we deduce that if $v_2(\beta_2)=m-1$ then $o_2=0<o_1$, and  so $2^{o_1} \mid x_2$, by \Cref{LemmaEnB2}.
Then $o_1+1-n_2=n_1-n_2=v_2(x_2) \geq o_1$, which implies $n_2=1$.
This yields the last congruence of the lemma.
\end{proof}

\begin{lemma} \label{tt-+}
Suppose that $\bb\in \B_r$, $\sigma_1=-1$ and $\sigma_2=1$.
\begin{enumerate}
	\item If $m\le n_2$ then
	\begin{itemize}
	 \item $o'_1(\bb)=o'_1(b)\le 1$;
	 \item $o'_2(\bb),o'_2(b)\le 1$, and
	 \item $o'_2(\bb)\ne o'_2(b)$ if and only if $v_2(x_2)=n_1-n_2$ and $o'_1(b)=1$.
	\end{itemize}
	\item If $m>n_2$ then
	\begin{itemize}
	 \item $o'_1(b),o'_1(\bb)\le 1$;
	 \item $o'_1(b)\ne o'_1(\bb)$ if and only if $2\nmid y_1$ and $n_1=o_1+1$, and
	 \item the following congruence holds
	\begin{equation}
	(u_2(\bb) -u_2(b))2^{n_2-1}+A \equiv  x_2  2^{m-o_1'(b)+n_2-n_1}   +B \bmod 2^m  \label{t2-+n2<}
	\end{equation}
	where
	$$A=\begin{cases}
	((x_1-1)y_2-x_2y_1) 2^{m-o_1+o_2-1},&\text{if } o_2+1=n_2\text{ and }o_1>0; \\
	0,&\text{otherwise};
	\end{cases}$$
	and
	$$B  = \begin{cases}
	x_2 2^{m-n_1}, & \text{if } n_1=o_1+1,  \ o_2=0<o_1, \text{ and } n_2=1;  \\
	0,&\text{otherwise}.
	\end{cases} $$
	\end{itemize}
\end{enumerate}
\end{lemma}

\begin{proof}
By \Cref{Fijando-r}, $p=2$, $n_1>n_2$ and hence, by \Cref{LemmaEnB2}.\eqref{EnB0EnB2}, $2\mid x_2$ and $2\nmid x_1y_2$.
Moreover, by \Cref{oo'}.\eqref{Cotao'}, $o'_1(b),o'_1(\bb)\le 1$ and $u_1(b)=u_1(\bb)=1$, so $t_1(b)=2^{m-o'_1(b)}$ and $t_{1}(\bb)=2^{m-o'_{1}(\bb)}$.
	
(1) Suppose first that $m\le n_2$.
Then $n_1>n_2\ge m>o_1$, and hence, by \Cref{ttbb-+}, $\beta_1\equiv \beta_2\equiv 0 \bmod 2^m$.
By   \eqref{Relpt2r1} and \Cref{PropEse}.\eqref{ValEse} also
	$t_2(\bb)(r_1-1)\equiv \Ese{r_2}{p^{n_2}}\equiv 0 \bmod p^m$,  therefore
	$m\le v_2(t_2(\bb))+v_2(r_1-1)=m-o'_2(\bb)+1$, and consequently $o'_2(\bb)\le 1$, $t_2(\bb)=p^{m-o'_2(\bb)}$ and $t_2(b) 2^{n_1-n_2}  \equiv 0 \bmod 2^m$.
	Then \Cref{PreCongruencias} and \Cref{ttbb-+} imply $o'_1(\bb)=o'_1(b)$ and $2^{m-o'_2(\bb)}\equiv x_2 2^{m-o'_1(b)+n_2-n_1}+2^{m-o'_2(b)} \bmod 2^m$.
	As $o'_2(\bb),o'_2(b)\le 1$ and $2^{n_1-n_2}\mid x_2$, it follows that $o'_2(\bb)\ne o'_2(b)$ if and only if $v_2(x_2)=n_1-n_2+o'_1(b)-1$  if and only if $v_2(x_2)=n_1-n_2$ and $o'_1(b)=1$.

(2) Suppose now that $m>n_2$. 
Then $o_2'(\bb)=o_2'(b)=m-n_2+1>1$, by \Cref{oo'}.\eqref{oo'SigmasDistintas}.  Thus \Cref{ttbb-+} yields that  \eqref{pret2}  takes the form of \eqref{t2-+n2<}.

We claim that
\begin{equation} \label{B1-+<n2}
y_1t_2(b) 2^{n_1-n_2} +y_12^{n_1-1}  \equiv \begin{cases}
y_1 2^{m-1}\bmod 2^m,&\text{if } n_1=o_1+1;\\
0\bmod 2^m,& \text{otherwise}.
\end{cases}
\end{equation}
Indeed, $y_1 t_2(b) 2^{n_1-n_2} +y_12^{n_1-1}\equiv y_12^{n_1-1} (u_2(b)  +1) \bmod 2^m $ and, by \Cref{oo'}.\eqref{oo'SigmasDistintas}, $u_2(b) (-1 +2^{m-o_1-1}) \bmod 2^{m-n_2}$. Thus  $v_2(u_2(b)+1)\geq \min(m-n_2, m-o_1-1)$.
Moreover $n_1>n_2$, and therefore $n_1-1+m-n_2\geq m$ and $n_1-1+m-o_1-1\geq m-1$. Hence, from \Cref{oo'}.\eqref{oMenorn}, it follows that  $n_1-1+v_2( u_2(b)+1)=m-1$ if and only if $n_1=o_1+1$. Therefore \eqref{B1-+<n2} follows.
 
By \eqref{pret1} and \Cref{ttbb-+}, and having in mind that $o'_1(b),o'_1(\bb)\le 1$, $2\nmid x_1y_2$ and $2\mid x_2$, we deduce that the expression in \eqref{B1-+<n2} is congruent modulo $2^m$ to $2^{m-o'_1(\bb)}-2^{m-o'_1(b)}$.
Hence, $o'_1(b)\ne o'_1(\bb)$ if and only if $n_1=o_1+1$ and $2\nmid y_1$, as desired. 
\end{proof}

\begin{lemma}\label{tt--}
If $\bb\in \B_r$ and $\sigma_2=-1$ then 
	\begin{eqnarray}\label{t1--}
	2^{m-o_1'(\bb)} &\equiv & x_1 2^{m-o_1'(b)} +y_1 2^{m-o_2'(b)+n_1-n_2}   
\bmod 2^m, \\
	2^{m-o_2'(\bb)} & \equiv & x_2  2^{m-o_1'(b)+n_2-n_1} +y_2 2^{m-o_2'(b)}+ B\bmod 2^m, \label{t2--}
	\end{eqnarray}
	where $B$ is as in \Cref{tt-+}.
\end{lemma}

\begin{proof}
Since $\sigma_2=-1$, by \Cref{Fijando-r}, \Cref{oo'}.\eqref{Cotao'} and 
\Cref{LemmaEnB2}, $p=2$, $\sigma_1=-1$, $o'_i(b),o'_i(\bb) \leq 1$, $u_i(b)=u_i(\bb)=1$ and $x_1\equiv y_2 \not\equiv x_2\equiv y_1\bmod 2$.
Then $t_i(\bb)(r_i-1)\equiv t_1(\bb)(r_2-1)\equiv t_2(\bb)(r_1-1) \equiv 0\bmod 2^m$, so
$t_i (\bb)\alpha\equiv (x_1y_2-x_2y_1)2^{m-o'_i(\bb)} \equiv 2^{m-o'_i(\bb)}\bmod 2^m$ and $t_i(b)=2^{m-o'_i(b)}$. Thus, by \Cref{PreCongruencias} it suffices to prove that $\beta_1\equiv 0 \mod 2^m$ and $\beta_2\equiv B \mod 2^m$.

By \Cref{EseProp}.\eqref{ValEse-1} and \Cref{ValEse2}
$$v_2(\beta_1)
\equiv \begin{cases} 
v_2(r_1+1) +v_2(x_1) +n_1-2=m-o_1 + v_2(x_1) +n_1-2 \ge m+v_2(x_1)-1, & 
\text{if } 2\mid x_1; \\ 
v_2(r_2+1) +v_2(y_1) +n_1-2=m-o_2 + v_2(y_1) +n_1-2 \ge m+v_2(y_1)-1, & 
\text{otherwise}.
\end{cases}$$
In both cases $v_2( \beta_1) \geq m$, so that $\beta_1\equiv 0 \bmod 2^m$.

Arguing similarly, in combination with \Cref{LemmaEnB2}, we obtain
$$v_2(\beta_2)
\equiv \begin{cases} 
m-o_1 + v_2(x_2) +n_2-2\ge m-o_1+n_1-2\ge m-1, & \text{if } 2\mid x_2; \\ 

m-o_2 + v_2(y_2) +n_2-2\ge m+v_2(y_2)-1\ge m, & \text{otherwise}.
\end{cases}$$
Thus either $\beta_2\equiv 0\bmod 2^m$ or $\beta_2\equiv 2^{m-1} \bmod 2^m$, $v_2(x_2)=n_1-n_2>0$ and $n_1=o_1+1$.
In the latter case $o_2<o_2+n_1-n_2=o_2+o_1+1-n_2\le o_1$, by \Cref{Cotasoo'}.\eqref{CotaOs}. Hence $o_2=0<o_1$, by \Cref{Fijando-r}.\eqref{CondicionesOs}, and $o_1+1-n_2=n_1-n_2=v_2(x_2)\ge o_1$, by \Cref{LemmaEnB2}, so that $n_2=1$. Thus $\beta_2 \equiv 2^{m-1} \bmod 2^m$ if and only if $n_1=o_1+1$, $o_2=0<o_1=v_2(x_2)$ and $n_2=1$. This proves that $\beta_2\equiv B\bmod 2^m$.
\end{proof}

\section{Description of $\B_r'$}\label{SectionBr'}

The aim of this section is to describe the set $\B_r'$.
Recall that $\B'_r$ is formed by the bases $b$ in $\B_r$ such that $o'(b)=(o'_1,o'_2)$, or equivalently $(|b_1|,|b_2|)\ge_{\lex} (|\bb_1|,|\bb_2|)$ for every $\bb\in\B_r$.

All throughout the section $b=(b_1,b_2)$ is a fixed element of $\B_r$ and the goal   is to obtain necessary and sufficient conditions for $b\in \B_r'$ in terms of conditions on the entries of $o'(b)=(o'_1(b),o'_2(b))$  without using the value of $o'_1$ and $o'_2$. The arguments in the proof of \Cref{OPrima++}, \Cref{OPrima-+}, and \Cref{OPrima--} are repetitive, so with the aim of avoiding boring repetitions we explain now the structure of the proofs. In all the cases we have to prove that $b\in \B_r'$ if and only $o'(b)$ satisfies certain conditions.
We also have a generic element $\bb=(b_1^{x_1}b_2^{y_1}[b_2,b_1]^{z_1},b_1^{x_2}b_2^{y_2}[b_2,b_1]^{z_2})$ in $\B_r$ and we compare $o'(b)$ and $o'(\bb)$ using the results of the previous section.
For the direct implication we assume that $b\in \B_r'$ and select several $\bb\in \B_r$ to deduce the conditions on $o'(b)$ from the inequality $o'(b)\ge_{\lex}o'(\bb)$. For the reverse implication we assume that $b\not\in \B_r'$ and  $\bb\in\B_r'$ and deduce from $o'(b)<o'(\bb)$ that 
$o'(b)$ does not satisfy the given conditions. In both implications it is important to have \Cref{LemmaEnB2} in mind because it establishes when $\bb$ belongs to $\B_r$.
Moreover, $A$ and the  $B$'s are as in \Cref{tt++}, \Cref{tt-+} or \Cref{tt--}, depending on the case considered, and always relative to the $\bb$ used in each case.

The following straightforward observation will be used repetitively: 

\begin{remark}\label{XYZWT}{\rm
Let $X$, $Y$, $Z$,  $W$  and $T$ be integers.
\begin{enumerate}
	\item \label{tf}
	If $v_p(X), v_p(Y)\leq m$ and $Y\equiv XW+Z\bmod p^m$ then
		$$v_p(X)\leq v_p(Y)  \quad \text{if and only if}\quad v_p(X)\leq v_p(Z)$$
	\item\label{tf2} If $v_2(X), v_2(Y)\leq m$ and $Y\equiv XW+Z+T2^{m-1}\bmod 2^m$ then
	$$v_2(X)\leq v_2(Y) \quad \text{if and only if} \quad \text{one of the following holds:} \begin{cases}
	2\mid T \text{ and } v_2(X)\leq v_2(Z); \\
	2\nmid T, \ v_2(X)\leq m-1 \text{ and } v_2(X)\leq v_2(Z);\\
	2\nmid T, \ v_2(X)=m  \text{ and }v_2(Z)= m-1. 
	\end{cases}$$
\end{enumerate}
}\end{remark}

\begin{lemma}\label{OPrima++} 
Suppose that $\sigma_1=1$ and $b\in \B_r$. 
Then $b  \in \B_r'$ if and only if the following conditions hold: 
	\begin{enumerate}
		\item If $o_1 =0$ then one of the following conditions holds:
		\begin{enumerate}
			\item $o'_1(b)\le o'_2(b)\le o'_1(b)+o_2 +n_1-n_2$ and 
			$\max(p-2,o'_2(b),n_1-m)>0$.
			\item $p=2$, $m=n_1$, $o_2'(b)=0$ and $o_1'(b)=1$.
		\end{enumerate}	
		\item If  $o_2 =0<o_1$ then $\max(p-2,o'_1(b),n_1-m)>0$ and
		$$o'_1(b)+\min(0,n_1-n_2-o_1 )\le o'_2(b)\le o'_1(b)+n_1-n_2.$$
		\item If $o_1 o_2 \ne 0$ then $o'_1(b)\le o'_2(b)\le o'_1(b)+n_1-n_2 $.
	\end{enumerate}
\end{lemma}

\begin{proof} 
By \Cref{Fijando-r}.\eqref{sigma1} $\sigma_2=1$, and by \Cref{tt++} the congruences \eqref{t1++} and \eqref{t2++} hold for each $\bb\in \B_r$.
We will use this without explicit mention. We consider separately three cases:

\textbf{Case (i)}. Suppose that $p>2$ or neither $0<m-n_2=o_1-o_2$ 
nor $m=n_1$ hold. 
Then $A=B_1=B_2=0$.
We now apply \Cref{XYZWT}.\eqref{tf} to \eqref{t1++}, with $Y=(x_1y_2-x_2y_1) u_ 1(\bb) p^{m-o_1'(\bb)}$, $X=u_1(b) p^{m-o_1'(b)}$ and $Z=y_1u_2(b)p^{m-o_2'(b)+n_1-n_2}$,  and to  \eqref{t2++}, with
$Y=(x_1y_2-x_2y_1)  u_2(\bb) p^{m-o_2'(\bb)}$, 
$X= u_2(b) p^{m-o_2'(b)}$ and 
$Z=x_2u_1(b)p^{m-o_1'(b)+n_2-n_1}$.
It follows that
$o'(\bb)\le_{\lex} o'(b)$  if and only if
	\begin{equation}
	\label{des1++} o_2'(b) \leq v_2(y_1) +n_1-n_2  +o_1'(b)
	\end{equation}
	and 
	\begin{equation}\label{des2++}
	\text{if } o_1'(b)=o_1'(\overline b) \text{ then } o_1'(b) \leq v_2(x_2) +n_2-n_1 +o_2'(b).
	\end{equation}  
	
Suppose that $b\in \B_r'$ and consider several $\bb\in \B_r$.
In case $o_1=0$, we first take $\overline b=(b_1 b_2^{p^{o_2}}, b_2) \in \B_r$, so \eqref{des1++} take the form $o_2'(b)\leq o_2+n_1-n_2 +o_1'(b)$ and then we  take 
$\overline b=(b_1,b_1^{{p^{n_1-n_2}}} b_2)\in \B_r$ so that \eqref{des2++} yields $o_1'(b) \leq o_2'(b)$. Thus condition (1) holds. 
In case $o_2=0<o_1$ we first take $\overline b=(b_1b_2, b_2)\in \B_r$ so that 
\eqref{des1++} gives that $o_2'(b)\leq n_1-n_2+o_1'(b)$, and then we take 
$\overline b=(b_1, b_1^{p^{\max(n_1-n_2,o_1)}}b_2)\in \B_r$ which using 
\eqref {des2++} yields $o_1'(b) \leq \max(0, o_1-n_1+n_2)+o_2'(b)$. 
Thus condition (2) holds. 
Finally, if $o_1o_2\neq 0$, $o_2<o_1<o_2+n_1-n_2$ by \Cref{Fijando-r}.\eqref{CondicionesOs}.  In this case we take $\overline b=(b_1^{1-p^{o_1-o_2}}b_2,b_2)\in \B_r   $, and \eqref{des1++} yields $o_2'(b)\leq n_1-n_2 +o_1'(b)$, and then we take $\overline b=(b_1, b_1^{p^{n_1-n_2}} b_2^{1-p^{n_1-n_2-o_1+o_2}})\in \B_r$ so that \eqref{des2++} implies $o_1'(b) \leq o_2'(b)$. Hence condition (3) holds.
	
Conversely assume $b\notin \B_r'$ and suppose that our generic element $\bb$ belongs to $\B_r'$ so that $o'(b)<o'(\bb)$. Thus either \eqref{des1++} or \eqref{des2++} fails. 
If  \eqref{des1++} fails then $o_2'(b)>v_2(y_1)+n_1-n_2+o_1'(b)\geq n_1-n_2+o_1'(b)$ and hence $b$ does not satisfy the consequent of neither condition (2) nor condition (3).
Thus we may assume that $o_1=0$.
By \Cref{LemmaEnB2}, $2^{o_2} \mid y_1$, so $o_2'(b)>o_2+n_1-n_2+o_1'(b)$, and hence $b$ does not satisfy (1a).
As $p>2$ or $m\ne n_1$, it follows that $b$ does not satisfy (1b) either.
Thus, $b$ does not satisfy neither (1), (2) or (3).
Suppose otherwise that \eqref{des1++} holds but \eqref{des2++} does not. 
Thus $o_1'(b)=o_1'(\overline b)$ and $o_1'(b)>v_2(x_2) +n_2-n_1+o_2'(b)$.
By \Cref{LemmaEnB2}, $v_2(x_2)\ge n_1-n_2$ and hence
$o_1'(b)>v_2(x_2) +n_2-n_1+o_2'(b) \geq o_2'(b)$. 
In particular, as we are assuming that if $p=2$ then $m\ne n_1$, necessarily $b$ does not satisfy the consequent of neither condition (1) nor (3).
If $o_2=0<o_1$ then $2^{\max( n_1-n_2,o_1)}\mid x_2$, by \Cref{LemmaEnB2}, so $o_1'(b)>\max (0,o_1+n_2-n_1 )+o_2'(b)$. Hence $b$ does not satisfy condition (2). 
	
\textbf{Case (ii)}. Suppose that $p=2$ and $0<m-n_2=o_1-o_2$. In particular,  $o_1\neq 0$. By \Cref{oo'}.\eqref{oo'Sigmas1},
$o_2'(b)=o_2'(\bb)=2m-n_2-o_1=m-o_2 \geq o_1'(\bb)$ for every $\bb\in \B_r$. 
Therefore the first inequality in (3) holds and if $o_2=0$ then $o_1=m-n_2\leq n_1-n_2$ and hence $o'_1(b)+\min(0,n_1-n_2-o_1 )\le o'_2(b)$. 
Moreover, $o'(\bb)\ge o'(b)$ if and only if $o'_1(\bb)\ge o'_1(b)$.

Assume $n_1>m$, so that $B_i=0$. In this case we have to prove that $b\in\B_r'$ if and only if $o'_2(b)\le o'_1(b)+n_1-n_2$. 
Indeed, applying \Cref{XYZWT} to \eqref{t1++}, with suitable  $X,Y$ and $Z$, we deduce that $o_1'(\bb) \geq o_1'(b)$ if and only if \eqref{des1++} holds.    If $b\in \B_r'$  then, taking $\bb=(b_1^{1-p^{o_1-o_2}} b_2, b_2)\in \B_r$, we deduce that $o_2'(b) \leq o_1'(b) +n_1-n_2$. 
	 Conversely if $b \not\in \B_r'$ and $\bb\in \B_r'$ then 
$o_1'(\bb) >o_1'(b)$ and hence $o_2'(b)>v_2(y_1) +n_1-n_2+o_1'(b)\geq n_1-n_2+o_1'(b)$, as desired. 

Otherwise, i.e. if $n_1\le m$ then $n_1=m$, by \Cref{oo'}.\eqref{oo'Sigmas1}. 
In particular $o_2=0$, by \Cref{Cotasoo'}.\eqref{m=n1}. 
Moreover, applying \Cref{XYZWT} to \eqref{t1++}
\medskip

$o_1'(\bb) \leq o_1'(b)$ if and only if 
	\begin{equation}
	\text{one of the following conditions holds} \quad 
	\begin{cases} \label{des1++Especial}  
	2\mid x_1y_1 \text{ and }  o_2'(b) \leq v_2(y_1) +n_1-n_2  +o_1'(b); \\
	2\nmid x_1y_1, \  1\leq o_1'(b)\text{ and } o_2'(b) \leq  n_1-n_2  +o_1'(b);\\
	2\nmid x_1y_1, \ o_1'(b)=0\text{ and }  1 +n_1-n_2= o_2'(b).
	\end{cases}
	\end{equation} 
	However, the assumption $n_1=m>n_2$ implies that $2\nmid x_1$ by \Cref{LemmaEnB2}, and $2\le m-o_1=n_2-m+o'_2$, by \Cref{oo'}.\eqref{CotaOs} and since $p=2$.
	Thus the last case of \eqref{des1++Especial} does not hold and 
	$2\mid x_1y_1$ if and only if $2\mid y_1$.

	Assume $b\in \B_r'$. Then taking $\overline b=(b_1b_2,b_2)\in \B_r$ condition \eqref{des1++Especial} implies that $1\leq o_1'(b)$ and $ o_2'(b) \leq n_1-n_2+o_1'(b)$. Thus condition (2) holds. 
	Conversely, if $b\not\in \B_r'$ and $\bb\in \B_r'$ then $o_1'(\bb)>o_1'(b)$  so \eqref{des1++Especial} does not hold. If $2\mid y_1$  then $o_2'(b) 
	>v_2(y_1) +n_1-n_2+o_1'(b)\ge 1+n_1-n_2+o_1'(b)$, so that (2) does not hold. 
	Similarly, if $2\nmid y_1$ and $1\leq o_1'(b)$ then $o'_2(b)>n_1-n_2+o'_1(b)$, so that again (2) does not hold.

\textbf{Case (iii)}. Finally suppose that $p=2$ and $m=n_1$ but $0<m-n_2=o_1-o_2$ does not hold. 
 Then for every $\bb\in\B_r$, $A=0$ and by \Cref{oo'}.\eqref{oo'Sigmas1}, $o_1o_2=0$.
 
 Assume firstly that $n_2\geq m$. Then $n_2=m=n_1$, so $o_1=0$ by \Cref{Fijando-r}.\eqref{CondicionesOs}, and $B_i=x_iy_i2^{m-1}$ for every $\bb\in\B_r$.  
By \Cref{LemmaEnB2}, $y_1\equiv y_2-1\equiv 0 \bmod p^{o_2}$,
and \Cref{XYZWT} yields
\medskip

 $o'(\bb)\leq_{\lex} o'(b)$ if and only if one of the conditions in \eqref{des1++Especial} holds and 
	\begin{equation}\label{des2++Especial} 
	\text{if } o_1'(b)=o_1'(\bb) \quad \text{ then  one of the following holds }\quad \begin{cases}
	2\mid x_2y_2 \text{ and } o_1'(b) \leq v_2(x_2)    +o_2'(b);\\
	2\nmid x_2y_2, \  1\leq o_2'(b)\text{ and }o_1'(b) \leq o_2'(b); \\
	2\nmid x_2y_2, \ o_2'(b)=0\text{ and }1=o_1'(b).
	\end{cases}
	\end{equation}	
We have to prove that $b\in \B_r'$ if and only if condition (1) holds. 	
	Suppose $b\in \B_r'$. If $o_2=0$ then we can take $\overline b=(b_2,b_1) \in \B_r$, so \eqref{des1++Especial} implies  $o_2'(b) \leq o_1'(b)+o_2$; and if $o_2>0$ then, taking $(b_1b_2^{2^{o_2}}, b_2)\in \B_r$,  \eqref{des1++Especial} implies $o_2'(b) \leq o_1'(b)+o_2$. Moreover taking $\overline b=(b_1,b_1b_2)\in \B_r$ and using \eqref{des2++Especial} we obtain that either $1\leq o_2'(b)$ and $o_1'(b) \leq o_2'(b)$ or $o_2'(b)=0$ and $o_1'(b)=1$. Thus condition (1) holds.  Conversely, suppose $b\not\in\B_r'$  and take $\bb\in \B_r'$ so that $o'(\bb)>o'(b)$. Thus either none of the conditions of \eqref{des1++Especial} holds or \eqref{des2++Especial} does not hold.
	Suppose that \eqref{des1++Especial} does not hold. 
	If $2\mid x_1y_1$ then $o_2'(b)>v_2(y_1)+o_1'(b)\ge o_1'(b)+o_2+n_1-n_2$, so that (1) does not hold.  
	If $2\nmid x_1y_1$ then either $o'_2(b)>o'_1(b)$, or  $o_1'(b)=0$  and $1+n_1-n_2\neq o_2'(b)$.
	In any case condition (1) does not hold.
	Suppose that \eqref{des2++Especial} does not hold.
	Therefore $o'_1(\bar b)=o'_1(b)$ and the three conditions on the right part of \eqref{des2++Especial} fail.
	If $2\nmid x_2y_2$ then $o'(b)\ne (0,1)$ and hence (1b) fails; and moreover $o'_2(b)=0$ or $o'_1(b)>o'_2(b)$ so that (1a) fails too.
	Suppose that $2\mid x_2y_2$.
	Then $o'_1(b)>v_2(x_2)+o'_2(b)\ge o'_2(b)$ and hence (1a) fails.
	If moreover $2\mid x_2$ then $o'_1(b)>v_2(x_2)+o'_2(b)>o'_2(b)\ge 0$ so that (1b) fails too.
	So we assume that $2\nmid x_2$ and hence $2\mid y_2$.
	Then $2\nmid y_1$ by \Cref{LemmaEnB2}.
	If (1b) holds then \eqref{t1++} yields the following contradiction $2^{m-1}\equiv x_1(1+y_1)2^{m-1} \equiv 0 \mod 2^m$. This completes the case $m\le n_2$.

	Finally suppose that $m>n_2$. By assumption $0<m-n_2\neq o_1-o_2$.
	In particular $o_1\ge o_1+o'_2-m=m-n_2>0$, thus $o_2=0$. So we have to prove that $b\in\B_r'$ if and only if condition (2) holds.
	Moreover, by \Cref{oo'}.\eqref{oo'Sigmas1}, $o_2'(\bb)=2m-n_2-o_1$ for each $\bb\in \B_r$. Then $n_1-n_2-o_1=m-n_2-o_1=o_2'(\bb)-m\le 0$ and hence
	$o_1'(b) +\min(0,n_1-n_2-o_1) = o_1'(b) +o_2'(b) -m \leq o_2'(b)$.
	As in the previous case, the third condition of \eqref{des1++Especial} does not hold. 
	Therefore $o'(\bb) \leq_{\lex} o'(b)$  if and only if $o_1'(\bb) \leq o_1'(b)$ if and only if one of the first two conditions in \eqref{des1++Especial} holds. 
	Suppose $b\in \B_r'$. Then taking $\bb=(b_1b_2,b_2)\in \B_r'$  from \eqref{des1++Especial} we obtain  that condition (2) holds. Conversely, if $b\not\in\B_r'$ and $\overline b\in \B_r'$ then $o_1'(\bb)>o_1'(b)$ and consequently none of the conditions in \eqref{des1++Especial} hold. If $2\mid x_1y_1$ then $2\mid y_1$ and $o'_2(b)>1+n_1-n_2+o'_1(b)$, so that condition (2) fails. Otherwise either $o_1'(b)=0$ or $o_2'>n_1-n_2+o_1'$.
	In all the cases condition (2) fails. 
\end{proof}

\begin{lemma} \label{OPrima-+} 
	Suppose $\sigma_1=-1$, $\sigma_2=1$ and let $b\in \B_r$. Then $b\in \B_r'$ if and only if the following conditions hold:
	\begin{enumerate}
		\item If $m\leq n_2$ then   $o'_1(b)\le o'_2(b)$ or $o_2=0<n_1-n_2<o_1$
	
		\item If $m>n_2$ then   $o'_1(b)=1$ or $   o_1+1\neq n_1$. 
	\end{enumerate} 
\end{lemma}

\begin{proof}
By \Cref{Fijando-r}, $p=2$ and $n_1>n_2$ and by \Cref{oo'}, $o'_1(b), o'_1(\bb)\le 1$.
	
	(1) Assume that $m\leq n_2$. By \Cref{tt-+}, $o'_1(b)=o'_1(\bb)$ and $o'_2(b),o'_2(\bb)\le 1$. Moreover $o'_2(b)\ne o'_2(\bb)$ if and only if $v_2(x_2)=n_1-n_2$ and $o'_1(b)=1$.
	This implies that if $o'_1(b)\le o'_2(b)$ then $b\in \B_r'$ because if $o'_1(b)=1$ then $o'_2(b)=1$ and otherwise $o'_2(b)=o'_2(\bb)$ for every $\bb\in \B_r$.
	It also implies, in combination with \Cref{LemmaEnB2}, that if $o_2=0$ 
	and $n_1-n_2<o_1$  then $o'_2(b)=o'_2(\bb)$, so that in this case also $b\in \B_r'$. 
	Suppose otherwise that $o'_2(b)<o'_1(b)$ and either $o_2\ne 0$ or $o_1\le 
	n_1-n_2$. 
	Then $o'_1(b)=1$, $o'_2(b)=0$, $\bb=(b_1,b_1^{2^{n_1-n_2}}b_2)\in \B_r$ and $o'_2(\bb)> o'_2(b)$. Thus $b\not\in \B_r'$.

	(2) Assume that $m>n_2$. Then $o'_2(b)=o'_2(\bb)=  m-n_2+1$ by \Cref{oo'}.\eqref{oo'SigmasDistintas}. Moreover, by \Cref{tt-+}, $o'_1(b)\ne o'_1(\bb)$ if and only if
		$2\nmid y_1$ and $n_1=o_1+1$. This implies that if $o'_1(b)=1$ 
	then $b\in \B_r'$. 
	If $o_1+1\ne n_1$ then $o_1'(b)=o_1'(\bb)$  and as $o'_2(b)=o'_2(\bb)$, in this case $\B_r=\B_r'$ and in particular $b\in \B_r'$.
	Finally, if $o'_1(b)\ne 1$ and $o_1+1=n_1$ then, taking $\bb=(b_1^{1-2^{o_1-o_2}}b_2,b_2)\in \B_r$, we obtain that $o'_1(\bb)>o'_1(b)=0$. Therefore $b\not\in \B_r'$.
\end{proof}

\begin{lemma}\label{OPrima--} 
	Suppose that $\sigma_2=-1$ and let $b\in \B_r$.
	Then $b  \in \B_r'$ if and only if the following conditions hold:  
	\begin{enumerate}
		\item If  $o_1 \le o_2$  and $n_1>n_2$ then $o_1'(b) \leq o_2'(b)$.
		\item If  $o_1=o_2$ and $n_1=n_2$ then $o_1'(b)\geq o_2'(b)$ 
		\item If $o_2=0<o_1=n_1-1$ and $n_2=1$ then  $o'_1(b)=1$ or $o'_2(b)=1$.  
		\item If $o_2=0<o_1$ and either $n_1\ne o_1+1$ or $n_2\ne 1$, then $o'_1(b)+\min(0,n_1-n_2-o_1)\le o'_2(b)$.
		\item If $o_1 o_2 \ne 0$ and $o_1\ne o_2$ then $o'_1(b)\le o'_2(b) $.  
	\end{enumerate}
\end{lemma}

\begin{proof}
By \Cref{Fijando-r} and \Cref{oo'}.\eqref{Cotao'}, $p=2$, $\sigma_1=-1$ and $o'_i(b), o'_i(\bb)\le 1$ for $i=1,2$. Moreover,   \eqref{t1--} and \eqref{t2--} hold.  We consider separately two cases:

	\textbf{Case  (i)}. Suppose that $o_1+1=n_1$, $o_2=0<o_1$ and $n_2=1$. Then  the summand $B$ in \eqref{t2--} is  $ x_22^{m-n_1}=T2^{m-1}$ with $T=x_22^{n_2-n_1}\in \Z$. 
	Applying \Cref{XYZWT} to the congruences \eqref{t1--} and \eqref{t2--}, and having in mind that $o'_i(b)\le 1$  we obtain that $o'(\bb)\le_{\lex} o'(b)$ if and only if 
	\begin{equation}\label{des1--}
		o_2'(b) \leq v_2(y_1) +n_1-n_2  +o_1'(b);
		\end{equation}    and  
	\begin{equation}\label{des2--especial}
	\text{if } o'_1(b)=o'_1(\bb) \text{ and } v_2(x_2)=n_1-n_2 \text{ then } o'_1(b)=1 \text{ or } o'_2(b)=1.
	\end{equation}
	Actually \eqref{des1--} always holds because $v_2(y_2)+n_1-n_2+o'_1(b)\ge 
	n_1-n_2=o_1\ge 1\ge o'_2(b)$. 	
	Thus $o'(\bb)\le_{\lex}o'(b)$ if and only if \eqref{des2--especial} holds. 
	Moreover, by the assumptions,  none of the antecedents  in 
	(1), (2), (4)  and  (5) holds,   while  the antecedent of   (3)  holds. 
	So we have to prove that $b\in\B_r'$ if and only if $o'_1(b)=1$ or $o'_2(b)=1$.
	Indeed, if $b\in \B_r'$ then taking $\bb=(b_1,b_1^{2^{o_1}} b_2)\in \B_r$ we obtain that $o'_1(b)=1$ or $o'_2(b)=1$ by \eqref{des2--especial}. Conversely, if $b\not\in\B_r'$ and $\bb\in \B_r'$ then \eqref{des2--especial} fails, so that $o'_1(b)=o'_2(b)=0$ and hence condition (3) fails. This finishes the proof in Case (i).

	\textbf{Case (ii)}. 
	Assume that one of the following conditions holds: 
	$o_1+1\ne n_1$, $o_2\ne 0$, $o_1=0$ or $n_2\ne 1$.  
	
	Then  the summand $B$ in \eqref{t2--} is  $ 0$ and by \Cref{XYZWT}, $o'(\bb)\le_{\lex}o'(b)$ if and only if \eqref{des1--} holds
	and 
	\begin{equation}\label{des2--}
	\text{if } o'_1(b)=o'_1(\bb) \text{ then } o_1'(b) \leq v_2(x_2) +n_2-n_1 +o_2'(b).
	\end{equation}  
	By hypothesis, the antecedent of  (3)  does not hold in this case. So we have to prove that $b\in \B_r'$ if and only if the conditions   (1), (2), (4) and (5)  hold. 
	
	Suppose that $b\in \B_r'$. 
	Assume $o_1\leq o_2$ and $n_1>n_2$.
	Then either $o_1=0$  or $0<o_1=o_2$  by \Cref{Fijando-r}.\eqref{CondicionesOs}. 
	In the first case   take    $\overline b=(b_1,b_1^{2^{n_1-n_2}}b_2  )\in \B_r$, and in the second  take  $\overline b=(b_1,b_1^{2^{n_1-n_2}}b_2^{1-2^{n_1-n_2}}  )\in \B_r$. Either way $o_1'(b)\leq o_2'(b)$  by \eqref{des2--}.
	Thus condition (1) hold. 
	If  $o_1=o_2$ and $n_1=n_2$ then, taking $\overline b=(b_2,b_1)$, \eqref{des1--} yields $o_1'(b)\geq o_2'(b)$.  Hence condition (2) holds.
	If $o_2=0<o_1$ then, taking $\overline b=(b_1,b_1^{2^{\max(n_1-n_2,o_1)}} b_2)\in \B_r$, \eqref{des2--} yields $o'_1(b) +\min(0,n_1-n_2-o_1 )\le o'_2(b) $.
	Thus condition (4) holds.
	If $o_1o_2\ne 0$ and $o_1\ne o_2$ then $o_2<o_1<o_2+n_1-n_2$, by \Cref{Fijando-r}.\eqref{CondicionesOs}. Then  taking, $\overline b=(b_1,b_1^{2^{n_1-n_2}} b_2^{1-2^{n_1-n_2+o_2 -o_1 }}) \in \B_r$, \eqref{des2--} yields $o_1'(b) \leq o_2'(b)$.

	Conversely, suppose that $b$ verify   (1), (2), (4) and (5)  and $b\not\in\B_r'$. Take $\overline b \in \B_r'$. 
	Then either \eqref{des1--} or \eqref{des2--} fails.
	If  \eqref{des1--} fails then $1\geq o_2'(b)> v_2(y_1) +n_1-n_2 +o_1'(b)$, 
	so necessarily $o_2'(b)=1$ and  $v_2(y_1)=o_1'(b)=n_1-n_2 =0$. 
	Then   either $o_1=0$ or $o_1=o_2> 0$ by \Cref{Fijando-r}.\eqref{CondicionesOs}.
	If $o_1=0$ then $ o_2\le  v_2(y_1)=0$  by \Cref{LemmaEnB2}. Hence  either way  $o_1=o_2$ and    condition   (2)  yields the contradiction $0=o_1'(b) \geq o_2'(b)=1$.   
	Suppose  that  condition \eqref{des2--} does not hold. 
	Then $o_1'(b)=o_1'(\bb)$ and $1\ge o_1'(b) > v_2(x_2) +n_2-n_1+o_2'(b) \ge 0 $.
	Therefore $o'_1(b)=1$, $o'_2(b)=0$ and $v_2(x_2)=n_1-n_2$. 
	If    $o_1\leq o_2$  then $n_1=n_2$ by condition  (1), hence $2\nmid x_2y_1$ and $2\mid x_1,y_2$ by \Cref{LemmaEnB2}.\eqref{B2Sigma}, therefore \Cref{tt--} yields the contradiction $0= o_2'(b)= o_1'(\bb)=o_1'(b) =1$.  
	Thus $o_2< o_1$.
	By condition  (5),  $o_1o_2=0$, so necessarily $o_2=0<o_1$. Then $n_1-n_2=v_2(x_2) \ge  o_1$ by \Cref{LemmaEnB2}, and consequently
	$\min(0,n_1-n_2-o_1) +o'_1(b)=o'_1(b)>o_2'(b)$ in contradiction with condition (4).  
\end{proof}

\section{Description of $\B_{rt}$}\label{SectionBrt}

In this section we describe $\B_{rt}$, or more precisely we describe the conditions that an element $b\in \B'_r$ must satisfy to belong to $\B_{rt}$.
We consider separately the cases $\sigma_1=1$ and $\sigma_1=-1$ and use the following notation from the main theorem:
\begin{eqnarray*}
a_1 &=& \min(o'_1,o_2,o_2+n_1-n_2+o'_1-o'_2); \\
a_2 &=&   \begin{cases}
	0, &\text{if } o_1=0; \\
	\min(o_1,o'_2,o'_2-o'_1+\max(0,o_1+n_2-n_1)), & \text{if } o_2=0<o_1; \\ \min(o_1-o_2,o'_2-o'_1), & \text{otherwise.} \end{cases}
\end{eqnarray*}
 
\begin{lemma}\label{ues+}
	Assume $\sigma_1=1$ and let $b\in \B_r'$. Then $b\in \B_{rt}$ if and only if  $u_1(b)\le p^{a_1}$ and  one of the following holds:\begin{enumerate}
		\item $u_{ 2 }(b)\leq p^{a_{ 2 }}$;
		\item $o_1o_2\neq 0$, $n_1-n_2+o_1'-o_2'=0<a_1$,     $1+p^{a_2}\leq u_2(b)\leq 2p^{a_2}$, and $u_1(b)\equiv 1\bmod p$.
	\end{enumerate}
\end{lemma}

\begin{proof}
By \Cref{oo'}.\eqref{oo'Sigmas1},  $o_2+o'_1\le m\le n_1$ and if $m=n_1$ then $o_1o_2=0$. Moreover, either $o_1+o'_2(\bb)\le m\le n_2$ or $2m-o_1-o'_2(\bb)=n_2<m$ for every $\bb\in \B_r$.
We will use this without specific mention.  
Let $\tilde \B_{rt}$  denote the  set of the elements in $\B_r'$ which  satisfy  $u_1(b)\le p^{a_1}$ and either  (1) or (2).

It suffices to prove the  following: 

\textbf{(i)} $\tilde \B_{rt}\neq \emptyset$.

\textbf{(ii)} If $b\in \tilde \B_{rt} $, $\bb\in\B_r'$ and $u(\bb)\le_{\lex} u(b)$  then   $u(b)=u(\bb)$.

\textit{Proof of (i)}.
 Start with $b=(b_1,b_2)\in \B_r'$.
 We construct another element $\overline{b}=( b_1^{x_1}b_2^{y_1}, b_1^{x_2}b_2^{y_2})$ with $x_1,y_1,x_2,y_2$ selected as in the Tables \ref{TCase1++}, \ref{TCase2++}   or  \ref{TCase3++}, depending on the values of $o_1$ and $o_2$.  The reader may verify that the conditions of the  \Cref{LemmaEnB2} hold, so   $\overline{b}\in \B_r$. 
 Then   using   congruences    \eqref{t1++} and \eqref{t2++} 
 we verify that in all the cases $o _i '(\bb)=o _i '$, which guarantees that 
 $\bb\in \B_r'$, and that $u_1(\bb)\le p^{a_1}$ and $u (\bb)$ satisfies either (1) or (2) in each case, i.e. $\overline{b}\in \tilde \B_{rt}$.

 (a) Suppose first that $o_1=0$. Write $u_1(b)=\rho +qp^{a_1}$ with $1\leq \rho \leq p^{a_1}$.
 Observe that $p\nmid \rho $, since $p\nmid u_1(b)$ and if $a_1=0$ then $\rho =1$.
 Let $\rho '$ be an integer with $\rho \rho '\equiv 1 \bmod p^m$.
 We select the $x$'s and $y$'s as in Table~\ref{TCase1++}. 

 \begin{table}[h!]
 	$$\matriz{{|l|cccc|}
 		o_1=0  & x_1 & y_1 &  x_2 & y_2  \\\hline
 		a_1=o'_1          & u_2(b) & 0       & 0 & 1 \\ 
 		a_1=o_2  & u_2(b)   & 0   & 0 & \rho 'u_1(b)
 		\\ 
 		a_1=o_2+n_1-n_2+o'_1-o'_2 < o_2 & u_2(b) & -qp^{o_2}  & 0 & 1 \\\hline}$$
 	\caption{\label{TCase1++}  }
 \end{table} 

We verify now that $o'(\bb)=o'$, $u_1(\bb)=\rho $ and $u_2(\bb)=1$, which imply that $\overline{b}\in \tilde \B_{rt}$, as desired.
Indeed, in all the cases $p \nmid y_2$.  Moreover \eqref{t2++} takes the form
 $$y_2 u_2(b)u_2(\bb) p^{m-o_2'(\bb)}\equiv  y_2 u_2(b) p^{m-o'_2 } \bmod
 p^m,$$
hence   $o_2'(\bb)=o_2'$ and  $u_2(\bb)=1$. Now we make use of \eqref{t1++}.
 If $o'_1= a_1$ then $ u_2(b) u_1(\bb) p^{m-o_1'(\bb)} \equiv  u_2(b)u_1(b) p^{m-o_1'}   \bmod p^m$ and hence   $o_1'(\bb)=o_1'$ and  $u_1(\bb)=u_1(b)  \le p^{o'_1}=p^{a_1}$, so that $u_1(\bb)=\rho $.
Suppose that $a_1=o_2$. Then $u_2(b)\rho 'u_1(b)u_1(\bb) p^{m-o_1'(\bb)}\equiv  u_2(b)u_1(b)p^{m-o'_1}\bmod p^m$.
 Thus $o_1'(\bb)=o_1'$ and $u_1(\bb)\equiv \rho  \bmod p^{o'_1}$.
 Since $1\le \rho ,u_1(\bb)\le p^{o'_1}$ we deduce that ${u}_1(\bb)=\rho $.
 Finally, assume that $a_1=o_2+n_1-n_2+o'_1-o'_2< o_2 $. 
 Then $o'_2>o'_1+n_1-n_2$ and hence $o_2\ne 0$, by \Cref{OPrima++}.(1).
 So $p\mid y_1$   and thus  $B_1\equiv 0 \bmod p^m$.
 Hence
 \begin{eqnarray*}
 u_2 (b) u_1(\bb) p^{m-o_1'(\bb)} &\equiv& u_2(b)u_1(b) p^{m-o_1'} -qu_2(b)p^{m-o_2'+o_2+n_1-n_2} \\
 &=& u_2(b)p^{m-o'_1}(u_1(b)-qp^{a_1}) = u_2(b)p^{m-o'_1}\rho  \bmod p^m.
 \end{eqnarray*}
Therefore $o_1'(\bb)=o_1'$ and, arguing as in the previous case, we deduce again $ {u}_1(\bb)=\rho $.

 (b) Suppose now that   $o_2=0<o_1$. In this case we write $u_2(b)=\rho +qp^{a_2}$ with $1\le \rho \le p^{a_2}$. Again $p\nmid \rho $ and we choose an integer $\rho '$ with $\rho \rho '\equiv 1  \bmod p^m$.
 Moreover let
 	$$\delta=\begin{cases} 1, & \text{if } p=2 \text{ and } m-n_2=o_1; \\
 	0, & \text{otherwise}.\end{cases}$$
 Then we take $x_1,x_2,y_1$, and $y_2$   as in \Cref{TCase2++}. We verify now that $o'(\bb)=o'$, $u_1(\bb)=1$ and $u_2(\bb)=\rho $, so that again $\overline{b}\in  \tilde \B_{rt}$.
 \begin{table}[h!]
 	$$\matriz{{|l|cccc|}
 		o_2=0<o_1  & x_1 & y_1 &  x_2 & y_2  \\\hline
 		a_2=	o'_2 & 1 & 0 & 0 & u_1(b)  \\
 		a_2=o'_2-o'_1+\max(0,o_1+n_2-n_1)<o_1 & 1 & 0 &  -qp^{\max(n_1-n_2,o_1)} & u_1(b) \\
 		a_2=o_1    & \rho  'u_2(b) +\delta q 2^{m-1} & 0 & 0 & u_1(b) \\
 		\hline}$$
 	\caption{\label{TCase2++}}
 \end{table}
 
 By \eqref{Cotao}, $m-1\ge o_1> 0$ and therefore in all the cases  $x_1\equiv 1 \bmod p^{o_1}$, so that the conditions in   \Cref{LemmaEnB2}  hold.
 By \eqref{t1++}, $ x_1u_1(b) u_1(\bb) p^{m-o_1'(\bb)}\equiv  x_1u_1(b) p^{m-o_1'} \bmod p^m$, so that $o_1'=o_1'(\bb)$ and $u_1(\bb)=1$.  Next we make use of \eqref{t2++}.
 If $  a_2=o_2'$ then $ u_1(b) u_2(\bb) p^{m-o_2'(\bb)}\equiv u_1(b)u_2 (b)  
 p^{m-o_2'} \bmod p^m$, so that $o_2'(\bb)=o_2'$ and $u_2(\bb)= u_2 (b)=\rho $.
  If $a_2=o'_2-o'_1+\max(0,o_1+n_2-n_1)<o_1$ then $n_1>m$, since otherwise      $m=n_1>n_2=2m-o_1-o_2'$, so $o_1\leq m-o_1'=  o'_2-o'_1+o_1+n_2-n_1 =a_2<  o_1$.
Thus $B_2=0$  and hence
	\begin{eqnarray*}
u_1(b) u_2(\bb) p^{m-o_2'(\bb)} &\equiv& -q p^{\max(0,o_1+n_2-n_1)} u_1(b) p^{m-o_1'}+ u_1(b) u_2(b)p^{m-o_2'} \\
&=& u_1(b)(-qp^{a_2}+u_2(b))p^{m-o'_2}= u_1(b)\rho p^{m-o'_2} \bmod p^m,
	\end{eqnarray*}
so that  once more $o_2'(\bb)=o_2'$ and $u_2(\bb)=\rho $.
Now assume  $a_2=o_1$.  If $p=2$ and $m-n_2=o_1$ then, having in
 mind that $1<v_2(r_1-1)=m-o_1=n_2$, we deduce that $A\equiv q2^{m-1}\bmod 2^m$.  Moreover,  $o_2'=o_2'(\bb) =2m-n_2-o_1 =m$.  Thus
 $$u_2(\bb)  \rho 'u_2(b)u_1(b)  \equiv q2^{m-1}+u_2(\bb) (\rho 'u_2(b)+q2^{m-1})u_1(b)\equiv  u_1(b)u_2(b)   \bmod 2^m.$$
 Otherwise $\delta=A=0$ and again $ \rho ' u_2 (b)  u_1(b)u_2(\bb) p^{m-o_2'(\bb)}\equiv  u_1(b)u_2(b) p^{m-o_2'}   \bmod p^m$.
 Either way $o_2'(\bb)=o_2'$ and $u_2(\bb)=\rho $, as desired.

 (c) Suppose that $o_1o_2\ne 0$  and let
	$$a'_1=n_1-n_2-\max(o_1-o_2,o'_2-o'_1)$$
Then $m<n_1$, by \Cref{oo'}.\eqref{oo'Sigmas1}; $n_2<n_1$, by  \Cref{Fijando-r}.\eqref{CondicionesOs}; $a_1=\min(o'_1,o_2)$, by \Cref{OPrima++}.(3); and $0\le \min(a_1',a_2)$ by the combination of \Cref{Fijando-r}.\eqref{CondicionesOs} and \Cref{OPrima++}.(3). 
Moreover, $a'_1=0$ if and only if $n_1-n_2=o'_2-o'_1$ and, in that case, $a_2=o_1-o_2>0$. Similarly, $a_2=0$ if and only if $o'_2=o'_1$ and, in that case, $a'_1=n_1-n_2-o_1+o_2>0$.  
	
Let $\rho $ and $q$ be integers such that
 $$1\le \rho \le p^{a_2} \qand
 u_2(b)=\rho +qp^{a_2}.$$
  Then define 
 $$(  R _2,q_2)=\begin{cases} 
 (\rho +p^{a_2}, q-1 ), & \text{if } u_1(b) \equiv qp^{a'_1} \bmod p \text{ and }0<a_1; \\
 (\rho ,q), & \text{otherwise}; \\
 \end{cases} 
 $$
 and
$$R= u_1(b)-q_2p^{a_1'}.$$ 
Finally, let $ R _1$ and $q_1$ be integers such that 
 $$1\le  R _1\le p^{a_1} \qand R= R _1 + q_1p^{a_1}.$$
As in the previous cases $p\nmid \rho $.

  \textbf{Claim}: 
$p\nmid R_i$ for $i=1,2$ and if $p\mid R$ then $a'_1=o'_1=0$. 

Indeed, as, $p\nmid u_1(b)$, if $p\mid R$ then $a'_1=0$ and hence $u_1(b)\equiv q_2 \bmod p$. Then $a_1=0$ by the definition of $q_2$ and thus $o'_1=0$.  This proves the last statement of the claim.
If $p\mid R_1$ then $a_1>0$ and hence $p\mid R$ so that also $a'_1=0$, and therefore $u_1(b)\equiv q_2 \bmod p$, in contradiction with the definition of $q_2$.
Finally, if $p\mid R_2$ then $u_1(b)\equiv qp^{a_1'}\bmod p$, $a_1>0$ and $a_2=0$ so that $a'_1>0$ and hence $p\mid u_1(b)$, yielding a contradiction.  This finishes the proof of the claim.

Therefore there are integers $ R' _1, R' _2 $ with $ R _i R' _i\equiv  1 \bmod p^m$, for $i=1,2$  and if $a'_1\ne 0$  or $o'_1\ne 0$
then there is another integer $R'$ such that $RR'\equiv 1 \bmod p^m$.
 Observe that $u_2(b)=R_2+q_2p^{a_2}$ and hence $R'_2u_2(b)\equiv 1 + R'_2q_2p^{a_2} \bmod p^m$.
 
We take    $x_1, x_2, y_1$ and $y_2$ as in  \Cref{TCase3++}  with 
$$\delta=\begin{cases} 1, & \text{if } p=2 \text{ and } m-n_2=o_1-o_2; \\
0, & \text{otherwise}.\end{cases}$$ 
In  all cases it is straightforward that  the  conditions of \Cref{LemmaEnB2} hold and that  $p\nmid y_2$.
 We will prove that $o'(\bb)=o'$ and $u_i(\bb)=R_i$ for $i=1,2$. 
 Then it is straightforward to verify that $\overline{b}\in \tilde \B_{rt}$, because if $R_2>p^{a_2}$ then $u_1(b)\equiv qp^{a'_1} \bmod p $, $0<a_1$, $q_2=q-1$ and $1+p^{a_2}\le R_2=\rho +p^{a_2}\le 2p^{a_2}$. Hence $a'_1=0$, so $u_1(b)\equiv q \bmod p$, and therefore $R_1=u_1(b)-(q-1) \equiv 1 \bmod p$.
 
 \begin{table}[h!]
 	$$\matriz{{|l|cccc|}
 		o_1o_2\ne 0    & x_1 & y_1 &  x_2 & y_2  \\\hline
 		a_2=o_2'-o_1' , \ a_1=o_2    & 1  &  0 & - R _1'q_2 p^{n_1-n_2}  &  R _1'u_1(b) \\ 
 		a_2=o_2'-o_1' , \ a_1=o_1'   & 1  &  0 & -R 'q_2p^{n_1-n_2} & R 'u_1(b) \\ 
 		a_2\ne o_2'-o_1', \ a_1=o_2  &  R _2'u_2(b) +\delta q_22^{o_2'-1}  & - R _2'q_2- \delta q_22^{o_2'-1-o_1 +o_2}&  0 &  R _1'R\\   
 		a_2\ne o_2'-o_1', \ a_1=o_1'   &  R _2'u_2(b)+ \delta q_22^{o_2'-1 } & 
 		- R'_2q_2 - \delta q_22^{o_2'-1-o_1 +o_2}&  0 & 1 \\   
 		\hline}$$ 
 	\caption{\label{TCase3++}}
 \end{table} 

 We will use \eqref{t1++} and \eqref{t2++}. 
Observe that $B_1=B_2=0$ because $m<n_1$. 
 
 Assume that $a_2=o_2'-o_1'$. 
 Then $a_1'=n_1-n_2+o_2-o_1>0$. Hence $u_1(b)\not\equiv q_2p^{a'_1} \bmod p$ so that $ R _2=\rho $ and
 $q_2=q$.
 Then $y_2\equiv  R' _1u_1(b)\bmod p^{o'_1}$.
 By \eqref{t1++}, $ y_2 u_1(\bb) p^{m-o_1'(\bb)}\equiv u_1 (b)p^{m-o_1'}\bmod p^m$, thus $o_1'(\bb)=o_1'$ and
 $R _1'u_1(b)u_1(\bb)\equiv u_1(b)\bmod p^{o_1'}$.
 This implies that  $R_1\equiv u_1(\bb) \bmod p^{o'_1}$. Hence, as $1\le R_1,u_1(\bb)\le p^{o'_1}$, we deduce that $u_1(\bb)= R _1$.
 Moreover, by \eqref{t2++},
	$$y_2 u_2(\bb) p^{m-o_2'(\bb)}\equiv
	x_2 u_1(b) p^{m-o_1' + n_2-n_1}+  y_2 u_2 (b)  p^{m-o_2'}\bmod p^m.$$
 Substituting $x_2$ and $y_2$ and multiplying in both sides by $ R _1$ if $a_1=o_2$ or by $R$ if $a_1=o_1'$, we obtain that  
	$$u_1(b) u_2(\bb) 2^{m-o_2'(\bb)}\equiv -  q_2 u_1(b) p^{m-o_1'} +u_2(b) u_1(b)p^{m-o_2'}\bmod p^m,$$
that is,
	$$u_2(\bb) p^{m-o_2'(\bb)}\equiv (  -q_2 p^{a_2}+u_2(b))p^{m-o_2'}= R _2 p^{m-o_2'} \bmod p^m.$$
Therefore $o_2'(\bb)=o_2'$ and $ u_2(\bb )=  R _2$.
 
Otherwise $0<a_2=o_1-o_2<o'_2-o'_1$ and  $a_1'=n_1-n_2+o_1'-o_2'$.
In particular, $a_2<o'_2$ and $o_1'<o_2'-1$.
Thus   $ \delta 2^{m-o_1'+o_2'-1} \equiv 0 \bmod p^m$.
 We consider separately the two options for $\delta$. 
 
 Suppose that  $\delta=0$.   Then  $A=0$   and by \eqref{t2++}
 $ y_2   R _2'u_2 (b)u_2(\bb)p^{m-o_2'(\bb)} \equiv 
  y_2 u_2(b) p^{m-o_2'}\bmod p^m$. Hence
 $o_2'(\bb)=o_2'$ and $   u_2(\bb)\equiv  R _2\bmod p^{o'_2}$.
 Moreover $1\le  R _2\leq 2p^{a_2}\leq p^{o_2'}$ and  $ 1\le  u_2 (\bb) \leq p^{o_2'}$, so that $u_2(\bb) = R _2$.    
 Furthermore, by \eqref{t1++},
	$$R_2'u_2 (b)y_2 u_1(\bb) p^{m-o_1'(\bb)}\equiv  R _2'u_2 (b)  u_1 (b)  p^{m-o'_1} -  R _2'q_2u_2(b)p^{m-o_2'+n_1-n_2}\bmod p^m,$$
hence
 \begin{equation}\label{EqRevision}
  y_2 u_1(\bb) p^{m-o_1'(\bb)}\equiv ( u_1(b)  - q_2 p^{a_1'}) p^{m-o_1'} \equiv Rp^{m-o_1'}\bmod p^m.
 \end{equation}
If $a_1=o_2$ then $o_1'\geq o_2>0$, so that $p\nmid R$, by the Claim.
Then \eqref{EqRevision} takes the form $R_1'R u_1(\bb) p^{m-o_1'(\bb)}\equiv R p^{m-o_1'}\bmod p^m$.
If $a_1=o_1'$ then $R=R_1+q_1 p^{o_1'}$ so that \eqref{EqRevision} takes the form $u_1(\bb) p^{m-o_1'(\bb)} \equiv R_1 p^{m-o_1'}\bmod p^m$.
Either way it follows that $o_1(\bb)=o_1' $ and $u_1(\bb) \equiv R_1 \bmod p^{o_1'}$; thus,   as  $R_1\leq p^{a_1}\leq p^{o_1'}$ and $u_1(\bb)\leq p^{o_1'}$,  we conclude that $u_1(\bb)=R_1$.

Finally assume $\delta=1$, i.e., $p=2$ and $m-n_2=o_1-o_2 $.  
Then  $n_2=2m-o_1-o_2'=2m-o_1-o_2'(\bb)$,  by \Cref{oo'}.\eqref{oo'Sigmas1},  so  that  $o_2'=o_2'(\bb)=m-o_2$. 
Since  $0<o_1-o_2<o'_2-o'_1$, $m-o_1'+o_2'-1>m$.
Furthermore, by \Cref{Fijando-r}.\eqref{CondicionesOs}, $m-1-o_1+o_2 +n_1-n_2\geq m$. Thus   congruence \eqref{t1++} implies
\begin{eqnarray*}
R _2'u_2(b) y_2 u_1(\bb) 2^{m-o_1'(\bb)} &\equiv&  R _2'u_2(b) u_1(b) 2^{m-o_1'}- R _2'q_2u_2(b) 2 ^{m-o_2'+n_1-n_2} \\
&=& R'_2u_2(b)2^{m-o'_1}(u_1(b)-q_22^{a'_1})=R'_2Ru_2(b)2^{m-o'_1} \bmod  2 ^m,
\end{eqnarray*}
i.e. $y_2u_1(\bb)2^{m-o'_2(\bb)}\equiv R2^{m-o'_1}\bmod 2^m$, and arguing as in  the previous paragraph we obtain again that  $o_1' =o_1'(\bb)$ and $u_1(\bb)=R_1$.
On the other hand $a_2+n_2-1=m-1$ and, as $p=2$ and $m\ge 2$, by \Cref{Cotasoo'}.\eqref{CotaOs}, $n_2+o_2'-2=2m-o_1-2\geq m$.
Thus
\begin{eqnarray*}
A&=& (R_2' u_2(b) +q_2 2^{o_2'-1} -1) y_2 2^{n_2-1} \\
&\equiv& ( R_2' q_2 2^{a_2} + q_2 2^{o_2'-1}) y_22^{n_2-1}
\equiv y_2 R_2' q_2 2^{a_2+n_2-1} + y_2q_2 2^{n_2+o_2'-2} \equiv q_2 2^{m-1}\bmod 2^m.
\end{eqnarray*}
Moreover $y_2\delta q_22^{o'_2-1}2^{m-o'_2}\equiv q_22^{m-1} \bmod 2^m $.
Hence, by \eqref{t2++},
$$R _2'u_2(b)y_2u_2(\bb) 2^{m-o_2'} \equiv  q_22^{m-1}+x_1y_2u_2(\bb)2^{m-o'_2(b)} \equiv y_2 u_2(b) 2^{m-o'_2} \bmod 2^m.$$
 So arguing as in the previous case  one obtains  that $ u_2(\bb)= R _2$.  
This finishes the proof of (i).

 \medskip
 
 \textit{Proof of (ii)}.  
Take $b \in \tilde \B_{rt}$ and $\overline b\in \B_r'$ such that  $u(\bb)\leq u(b)$. Then $o'_i(b)=o'_i(\bb)$ and $\overline b_i=b_1^{x_i}b_2^{y_i} [b_2,b_1]^{z_i}$ for some integers $x_i, y_i,z_i$ satisfying the conditions in \Cref{LemmaEnB2} and  congruences \eqref{t1++} and \eqref{t2++}, for $i=1,2$.

 (1) Suppose first that $o_1=0$.  
 Then  $a_2=0$ and  $u_i(b)   \leq p^{a_i}$. Thus   $1=u_2(b)=  u_2(\bb) $  and $ 1\le u_1(\bb) \le   u_1(b)\leq p^{a_1}$.  
 If   $a_1=0$  then  $1=u_1(b)=  u_1(\bb)$. Assume otherwise.    As $a_1\le o_2$, by \Cref{LemmaEnB2} we   have that $y_1\equiv y_2-1 \equiv 0 \bmod p^{ a_1 }$,  so $B_1\equiv 0\bmod p^{ a_1+n_1-1}$, in particular  $B_1\equiv 0\bmod p^m$, and hence  \eqref{t1++} implies $  u_1(\bb) x_1 \equiv x_1 u_1(b) \bmod  p^{a_1}$, while $x_1\equiv x_1y_2-x_2y_1\not\equiv  0\bmod p$, so $u_1(\bb)\equiv u_1(b) \bmod p^{a_1}$.
 Therefore $ u_1(\bb)=u_1(b)$.

 (2) Assume now that $o_2=0  <o_1$.  
 Then   $a_1=0$ and  $n_2<n_1$ by \Cref{Fijando-r}.\eqref{CondicionesOs}.
 Moreover, $(u_2(\bb),u_1(\bb)) \leq_{\lex} (u_2(b),u_1(b))$,  $u_2(b) \leq  p^{a_2}$  and  $u_1(b) =1$.
 Hence it suffices to prove that $u_2(b)=u_2(\bb)$. 
 Moreover   $p^{\max(o_1,n_1-n_2)} \mid x_2$. 
 We assert that $A\equiv B_2\equiv 0\bmod p^m$ or $a_2<o_2'$.
 Otherwise,  $p=2$ and, by the definition of $a_2$,  $o_2'\leq o_1$.
 If $A\not\equiv 0\bmod 2^m$ then $0 <  o_1=m-n_2=-m+o_1+o_2'$,  by \Cref{oo'}.\eqref{oo'Sigmas1},  so that $o_1<m=o_2'\le o_1$, a contradiction.
  If  $B_2\not\equiv 0 \bmod 2^m$  then  $m=n_1$ and $0 <  o_1\le v_2(x_2)\leq n_1-n_2$, by \Cref{LemmaEnB2}. Thus, again by \Cref{oo'}.\eqref{oo'Sigmas1}, $o_1\leq m-n_2=-m+o_1+o_2'$,  yielding  the contradiction $m \leq o_2'\leq o_1 <m $. This proves the assertion.
  If $A\equiv B_2 \equiv 0 \mod p^m$ then dividing by $p^{m-o'_2}$ in \eqref{t2++}, with the help of \Cref{LemmaEnB2} and \Cref{OPrima++}, the reader may verify that  $u_2(\bb) y_2\equiv y_2 u_2(b) \bmod p^{o'_2}$ and hence also $u_2(\bb) y_2\equiv y_2 u_2(b) \bmod p^{a_2}$. Otherwise, $p=2$ and $0\le a_2<o'_2$ and hence $m-o'_2\le \min(v_2(A),v_2(B_2))$. Thus the same argument shows that $u_2(\bb) y_2\equiv y_2 u_2(b) \bmod p^{a_2}$. Since $y_2\equiv x_1y_2-x_2y_1 \not\equiv 0 \bmod p$,  $u_2(\bb)\equiv u_2 (b)\bmod p^{a_2}$, so $u_2(b)=u_2(\bb)$  as desired.
 
 (3) Assume that $o_1o_2\neq 0$. Then by \Cref{LemmaEnB2},
$$x_1= 1+ x_1' p^{o_1-o_2}, \qquad   x_2=x_2' p^{n_1-n_2}, \qquad 
 y_1=-x_1' -y_1'p^{o_1}, \qquad   y_2=1-x_2'p^{n_1-n_2+o_2-o_1} +y_2' p^{o_2},
$$
 for  some integers $x_1',x_2',y_1',y_2'$. 
By \Cref{Fijando-r}.\eqref{CondicionesOs}, $0<\min (o_2,o_1-o_2, n_1-n_2+o_2-o_1)$, so clearly $p\nmid x_1 y_2$ and, by \Cref{OPrima++}.(3), $a_1=\min(o'_1,o_2)$.  Moreover, by \Cref{Cotasoo'}.(4),  $m<n_1$, so that $B_1=B_2=0$.  Thus \eqref{t1++} and \eqref{t2++} take the forms
 \begin{align*}
 u_1(\bb) \left(x_1(1+y_2' p^{o_2}) + x_2' p^{n_1-n_2-o_1+o_2}(y_1'  p^{o_1}-1) \right) \equiv  x_1 u_1(b)  - (x_1'+y_1' p^{o_2}) u_2 
 (b)p^{n_1-n_2+o_1'-o_2'} & \bmod p^{o_1'}, \\
 A p^{o_2'-m}+  u_2(\bb)\left (y_2 +  x_1' p^{o_1-o_2}(1+y_2'  p^{o_2}) +x_2'y_1' p^{n_1-n_2 +o_2} \right)  \equiv  u_1(b) x_2' p^{o_2'-o_1'}
 + u_2(b)y_2  & \bmod p^{o_2 '}.
 \end{align*}
The first congruence implies that $u_1(b)\equiv  u_1(\bb) \bmod p^{\min (a_1,a_1')}$  and the second one that $Ap^{o'_2-m}+y_2u_2(\bb)\equiv y_2u_2(b) \bmod p^{a_2}$.
Suppose that $u_2(\bb)\not\equiv u_2(b)\bmod p^{a_2}$.
Then $p^{a_2}\nmid Ap^{o_2'-m}$ and hence $p=2$ and $m-n_2=o_1-o_2>0$.
Thus $2m-o_1-o'_2=n_2$,  by \Cref{oo'}.\eqref{oo'Sigmas1}, so that $o_2+o'_2=m$ and hence $a_2\le o_1-o_2\le m-1-o_2=o'_2-1$. As $2^{m-1}\mid A$, it follows that $2^{a_2}\mid A2^{o'_2-m}$, a contradiction.
 Therefore,
$u_2(b)\equiv   u_2(\bb) \bmod p^{a_2}$.
If $a_1\leq a_1'$ then $n_1-n_2+o'_1-o'_2\ne 0$ or $a_1=0$ and hence $1\leq u_i(b) \leq p^{a_i}$.  
Therefore 
$u_2(\bb)=u_2(b)$ and $u_1(\bb)=u_1(b)$, and we are done.   

So we can assume $a_1'<a_1$. Fix integers $\lambda_1$ and $ \lambda_2$ such that $u_1(b)=  u_1(\bb)+\lambda_1 p^{a_1'}$ and $u_2(b)=  u_2(\bb)+\lambda_2 p^{a_2}$.
 So  the congruences above can be rewritten as  
	$$u_1(\bb) \left(  x_2' p^{n_1-n_2-o_1+o_2-a_1'}(y'_1p^{o_1}-1) \right.
	\left. + x_1 y_2' p^{o_2-a_1'} \right)
	\equiv  x_1 \lambda_1   - (x_1'+y_1' p^{o_2}) u_2(b) p^{n_1-n_2+o_1'-o_2'-a_1'} \bmod p^{o_1'-a_1'},$$
and
	$$A p^{o_2'-m-a_2}+  u_2(\bb)\left(  x_1' p^{o_1-o_2-a_2}(1+y'_{ { 2}}p^{o_2})  \right.
	 \left. + x_2'y_1' p^{n_1-n_2 +o_2-a_2} \right)   \equiv  u_1(b) x_2' p^{o_2'-o_1'-a_2 } + y_2 \lambda_2     \bmod p^{o_2 '-a_2}.$$
Note that $o_1'-a_1'< o_2'-a_2$, since $a_1'-a_2= n_1-n_2+o_ 2 -o_ 1  +o_1'-o_2' > o_1'-o_2'$.
 This,   together with $a'_1<a_1=\min(o_2,o'_1)$, $o_1-a_2\ge o_2$  and $n_1-n_2-a_2=n_1-n_2-\min(o_1-o_2,o'_2-o'_1)\ge a'_1\ge 0$ implies that
 \begin{align*}
 -	  u_1(\bb)     x_2' p^{n_1-n_2-o_1+o_2-a_1'}    &\equiv  (1+x_1' p^{o_1-o_2}) \lambda_1   -  x_1'  u_2(b) p^{n_1-n_2+o_1'-o_2'-a_1'} \bmod p^{a_1-a_1'} \\
 u_2(\bb)  x_1' p^{o_1-o_2-a_2}    &\equiv  u_1(b) x_2' p^{o_2'-o_1'-a_2 } +(1-x_2' p^{n_1-n_2+o_2-o_1})\lambda_2  \bmod p^{a_1-a_1'}.
 \end{align*} 
 If $a_2=o_1-o_2$, then $a_1'=n_1-n_2 +o_1'-o_2'$, and
 \begin{align*}
 -  u_1(\bb)     x_2' p^{o_2'-o_1'-o_1+o_2 }    &\equiv  (1+x_1' p^{o_1-o_2}) \lambda_1   -  x_1'  u_2(b)  \bmod p^{a_1-a_1'} \\
 u_2(\bb)  x_1'     &\equiv  u_1(b) x_2' p^{o_2'-o_1'-o_1+o_2 } + (1-x_2' 
 p^{n_1-n_2+o_2-o_1}) \lambda_2  \bmod p^{a_1-a_1'}.
 \end{align*}
 By substracting the second  congruence to the first one and simplifying, 
 we obtain that $$ 0\equiv (\lambda_1-\lambda_2)(1+x_1' p^{o_1-o_2} -x_2' p^{n_1-n_2-o_1+o_2}) \bmod p^{a_1-a_1'}.$$
 Similarly, if $a_2=o_2'-o_1'$ it must be $a_1'=n_1-n_2+o_2-o_1$, so 
 \begin{align*}
 -	  u_1(\bb)     x_2'     &\equiv  (1+x_1' p^{o_1-o_2}) \lambda_1   -  x_1'  u_2(b) p^{o_1-o_2+o_1'-o_2' } \bmod p^{a_1-a_1'} \\
 u_2 (\bb) x_1' p^{o_1-o_2-o_2'+o_1'}    &\equiv  u_1(b) x_2'   +(1-x_2' p^{n_1-n_2+o_2-o_1})\lambda_2  \bmod p^{a_1-a_1'},
 \end{align*}
 and consequently once again
 $$ 0\equiv (\lambda_1-\lambda_2) (1+x_1' p^{ {o_1 }- {o_2 }} -x_2'p^{n_1-n_2-o_1+o_2})  \bmod p^{a_1-a_1'} .$$
 Either way $\lambda_1 \equiv \lambda_2 \bmod p^{a_1-a_1'}$.
 Fix an integer $q_1$ such that  $\lambda_1 = \lambda_2 +q_1 p^{a_1-a_1'}$.
 
 Moreover $1\le u_2(\bb) \leq  u_2(b)\le 2p^{a_2}$ and, as $u_2(b)= {u}_2(\bb)+\lambda_2 p^{a_2}$ it follows that $\lambda_2\in \{0, 1\}$. 
 We claim that $\lambda_2=0$. 
 Otherwise   $1+p^{a_2}\leq u_2(b)\leq 2p^{a_2}$, 
 $1  \leq u_2(\bb) \leq  p^{a_2}$, $\lambda_2= 1$ and $u_1(b)\equiv 1 \bmod p$.   Therefore condition (2) holds,  and hence $u_1(b)=u_1(\bb)+\lambda_2 + q_1p^{a_1} \equiv  {u}_1(\bb)+1\bmod p$. Therefore $u_1(\bb)\equiv 0 \bmod p$,  in  contradiction with the fact that $  u_1(\bb)$ is not a multiple of $p$.
Therefore $\lambda_2=0$, so that $u_2(b)=u_2(\bb)$, $p^{a_1-a'_1}\mid \lambda_1$ and $u_1(b)\equiv u_1(\bb)\bmod p^{a_1}$, which implies that $u_1(b)= {u}_1(\bb)$ because $1\le {u}_1(\bb)\leq  u_1(b)  \le p^{a_1}$.
\end{proof}

 \begin{lemma}\label{ues-}   Suppose that  $\sigma_1=-1$ and let $b\in \B_r'$. Then $b\in \B_{rt} $ if and only if one of the following conditions holds:
  \begin{enumerate}
  	\item $\sigma_2=-1$ or $m\leq n_2$.
 	\item  $o'_1=0$ and either $o_1=0$ or $o_2+1\ne n_2 $.
 	\item $o_1'=1$, $o_2=0 $ and $ n_1-n_2<o_1$. 
 	\item $  u_2(b) \leq  2^{m-n_2}$.
	\end{enumerate}
 
 \end{lemma}
 
 \begin{proof} 
By \Cref{Fijando-r}.\eqref{ImparSigma1} and \Cref{oo'}.\eqref{oo'SigmasDistintas},   $o'_1\le 1$ and $u_1(b)=1$ for each $b\in \B_r'$.   Also by \Cref{oo'}.\eqref{oo'SigmasDistintas}, if $\sigma_{2}=-1$ or $m\le n_2$ then   $o_2' \leq 1$ and $u_2(b)=1$ for each $b\in \B_r'$. In  that case $\B_r'=\B_{rt}$, so we shall assume otherwise, i.e., $\sigma_2=1$  and $n_2<m$.
Then, once more by \Cref{oo'}.\eqref{oo'SigmasDistintas}, $n_2=m-o_2' +1 <m$ and   $u_2(b)\in \{v,v+2^{m-n_2}\}$, where $v$ is the unique integer satisfying $1\le v\le 2^{m-n_2}$ and $v(1+2^{m-o_1-1})\equiv -1\bmod  2^{m-n_2}$.

We argue as in the proofs in \Cref{SectionBr'}, namely we use several  $\bb=(b_1^{x_1}b_2^{y_1}[b_2,b_1]^{z_1},b_1^{x_2}b_2^{y_2}[b_2,b_1]^{z_2})\in \B'_r$ to compare $u(b)$ and $u(\bb)$ with the help of   \eqref{t2-+n2<}.
Now it is necessary to have  \Cref{LemmaEnB2} and  \Cref{OPrima-+}   in mind to verify that the different $\bb$ constructed belong to $\B'_r$.
  Moreover, $b\in \B_{rt}$ if and only if    $u_2(b)\le u_2(\bb)$  for every $\bb\in \B'_r$
Observe that (4)  is equivalent to   $u_2(b)=v$ and in that case obviously $u_2(b)\le u_2(\bb)$. 
Thus  we may assume that $u_2(b)\ne v$ and we have to prove that  $u_2(b)=u_2(\bb)$ for every $\bb\in\B'_r$  if and only if   either (2) or (3) holds.

Suppose that  $u_2(b)=u_2(\bb)$  for every $\bb\in\B'_r$. 
 We consider separately the two possible values of $o'_1$. 
Firstly   assume    that  $o_1' =1  $.  
  If $o_2=0<o_1$, $n_1=o_1+1$ and $n_2=1$   and we take  $\bb=(b_1^{1-2^{o_1}}, b_ 1^{2^{ o_1 }} b_2)\in \B_r$ then $o'_1(\bb)=o'_1(b)$, by \Cref{tt-+}.(2), so that   $\bb\in \B_r'$  by  \Cref{OPrima-+}.(2),  and  \eqref{t2-+n2<}   yields the contradiction  
$  2^{m-1}   \equiv 0  \bmod 2^m $,  since    $A=B=2^{m-1}$.
Therefore   $o_1+1<n_1$, $0<o_2$, $o_1=0$ or $1<n_2$. Then $B=0$.
If $o_1=0$ then  take $\bb= (b_1, b_1^{2^{n_1-n_2}} b_2) $; if $o_2=0 <o_1\leq n_1-n_2$  then  take 
$\bb=(b_1 ,b_1^{2^{n_1-n_2}}b_2) $, and if  $o_1o_2\neq 0$  then take $\bb=(b_1,b_1^{2^{n_1-n_2}}b_2^{1-2^{n_1-n_2-o_1+o_2}}) $.  
In any case $\bb\in \B'_r$ by \Cref{tt-+}.(2) and \Cref{OPrima-+}.(2)  and $A\equiv 0\bmod 2^m$  so that  congruence \eqref{t2-+n2<}    yields again the contradiction $2^{m-1}\equiv 0 \bmod 2^m$.   Therefore  $o_2=0 $ and $n_1-n_2<o_1$, i.e., condition (3) holds.
Next suppose $o_1'=0$. 
Then $o_1+1 \ne   n_1$ and $\B_r=\B_r'$ by \Cref{OPrima-+}. The former implies that $B=0$.  Suppose that  $o_2+1=n_2$ and $o_1>0$  and take $\bb=(b_1^{1-2^{o_1-o_2}} b_2, b_2)  $.    By \Cref{Fijando-r}, $o_1\ne o_2$ and hence $\bb\in \B_r'$.
Moreover, $A =  2^{m-1} $, so \eqref{t2-+n2<}   yields once more the contradiction  $  2^{m-1}\equiv 0\bmod 2^m$.
Therefore $o_1=0$ or $o_2+1\neq n_2$, i.e., condition (2) holds.
 	
Conversely, assume   $u_2(b)\ne u_2(\bb)$ for some $\bb\in \B_{rt}$. Therefore  $u_2(b)-u_2(\bb)=2^{m-n_2}$   and we have to prove that neither (2) nor (3)  does hold.   Suppose  that  $o_1'=1$, $o_2=0$ and $n_1-n_2<o_1$. In particular $1<n_2$, by \Cref{Cotasoo'}.\eqref{CotaOs}, and $2^{n_1-n_2+1}\mid x_2$,   by \Cref{LemmaEnB2}, so that $A=B=0 $.   Moreover, $n_1-n_2+1\le o_1\le v_2(x_2)$, by \Cref{LemmaEnB2}, and therefore $x_2 2^{m-o_1' +n_2-n_1} \equiv 0 \bmod 2^m$.
Thus \eqref{t2-+n2<} yields the contradiction $2^{m-1}\equiv 0\bmod 2^m$.   Suppose that   $o_1'=0$,  and either $o_2+1<n_2$ or $o_1=0$. This implies that $A=0$.
Since $o_1'=0$ and $m>n_2$, \Cref{OPrima-+} yields $o_1+1\ne  n_1$. Thus $B=0$. Hence once more \eqref{t2-+n2<} implies the contradiction $2^{m-1} \equiv 0\bmod 2^m$.
 \end{proof}

\section{Proof of the theorem}\label{SeccionDemostracion}

By the arguments given in the introduction, the map associating $\inv(G)$ to the isomorphism class of a 2-generated non-abelian cyclic-by-abelian group of prime-power order $G$ is well defined and injective. So to prove our main result it is enough to show that the image of this map is formed by the lists satisfying the conditions in the theorem.

We first prove that if $G$ is a $2$-generated group of prime-power order and  		
	$$\inv(G)=(p,m,n_1,n_2,\sigma_1,\sigma_2,o_1,o_2,o'_1,o'_2,u_1,u_2)$$ 
then the conditions in the theorem hold. Condition \eqref{1} follows from the definition of $p,m,n_1$ and $n_2$. 
Conditions \eqref{2}, \eqref{3} and \eqref{4} follow from the definition of $\sigma_i$ and $u_i$ and from \eqref{Cotao} and \Cref{Cotasoo'}.  
Condition \eqref{5} is a consequence of \Cref{Fijando-r}.

To prove \eqref{6} and \eqref{7} we fix $b\in \B_{rt} $,  which exists by \Cref{Fijando-r}.\eqref{bsConErres} and the definition of $\B_{rt}$.  Then  $o'_i=o'_i(b)$, $u_i=u_i(b)$.
Suppose that $\sigma_1=1$.  Then  \eqref{6a} follows from \Cref{Fijando-r}.\eqref{sigma1}  and \Cref{Cotasoo'}.\eqref{m>n1}; \eqref{6b} from \Cref{oo'}.\eqref{oo'Sigmas1}; \eqref{6c}, \eqref{6d} and \eqref{6e} follow from \Cref{OPrima++}; and \eqref{6f} and \eqref{6g} follow from \Cref{ues+}. 
This finishes the proof of condition \eqref{6}.
Suppose now that $\sigma_1=-1$. 
 Then  \eqref{7a}    follows from the definition of $\sigma_1$ and \Cref{Cotasoo'}.\eqref{Cotao'}. 
  Suppose that  $\sigma_2=1$. Then $n_2<n_1$, by \Cref{Fijando-r}.\eqref{nsIguales}.   
Therefore \eqref{T-+m>=} follows from \Cref{Cotasoo'}.\eqref{sigmasDistintosm>=} and \Cref{OPrima-+}.(1); and \eqref{T-+m<}  follows from \Cref{Cotasoo'}.\eqref{sigmasDistintosm<}, \Cref{OPrima-+}.(2) and \Cref{ues-}.
Finally,  \eqref{7c}  follows from \Cref{Cotasoo'}.\eqref{Cotao'} and  \Cref{OPrima--}.
This finishes the proof of condition \eqref{7}.
 
To complete the proof of the theorem we need the following lemma, where for integers $m$ and $n$ with $n>0$,  $\lfloor \frac{m}{n} \rfloor$ and $[m]_n$ denote, respectively, the quotient and the remainder of $m$ divided by $n$.

\begin{lemma}\label{Cociclo}
Let $M,N_1,N_2,r_1,r_2,t_1,t_2$ be positive integers satisfying the following conditions:
	\begin{eqnarray}
\label{Relacion-rN}   r_i^{N_i} & \equiv & 1 \bmod M, \\
\label{Relacion-tr}  t_ir_i & \equiv & t_i \bmod M, \\
\label{Relacion-t1r2} \Ese{r_1}{N_1} & \equiv & t_1(1-r_2) \;  \bmod M, \\
\label{Relacion-t2r1} \Ese{r_2}{N_2} & \equiv & t_2(r_1-1) \bmod M.
\end{eqnarray}
Consider the groups $A=\GEN{a}$ and $B=\GEN{b_1}\times \GEN{b_2}$ with $|a|=M$ and $|b_i|=N_i$ for $i=1,2$. 
Then there is a group homomorphism $\sigma:B\rightarrow \Aut(A)$ given by $a^{\sigma(b_i)} = a^{r_i}$ and a $2$-cocycle  $\rho:B\times B\rightarrow A$ 
	given by 	$$\rho(b_1^{x_1} b_2^{y_1},  b_1^{x_2}b_2^{y_2})= a^{ r_2^{y_2} \Ese{r_1}{x_2} \Ese{r_2}{y_1} +   t_1 r_2^{y_1+y_2}   \lfloor \frac{x_1+x_2}{N_1} \rfloor  + t_2 \lfloor \frac{y_1+y_2}{N_2} \rfloor } $$ for $b_1^{x_i}b_2^{y_i}\in B$ with $0\leq x_i<N_1 $ and $0\leq y_i<N_2$, for $i=1,2$.
\end{lemma}

\begin{proof}
The only non-obvious statement  is   that $\rho$ satisfies the cocycle condition, namely: 
\begin{equation}\label{cocycle}
\rho(b_1^{x_1} b_2^{y_1} ,b_1^{x_2+x_3} b_2^{y_2+y_3})   \cdot  \rho(b_1^{x_2} b_2^{y_2}, b_1^{x_3}b_2^{y_3}) = \rho(b_1^{x_1+x_2} b_2^{y_1+y_2}, b_1^{x_3} b_2^{y_3}) \cdot   \rho(b_1^{x_1} b_2^{y_1},b_1^{x_2}b_2^{y_2}) ^{\sigma(b_1^{x_3}b_2^{y_3}) }
\end{equation} 
for  $ b_1^{x_i}b_2^{y_i} \in B$ with  $0\leq x_i<N_1$ and $0\leq y_i<N_2$ for $i=1,2,3$. 
To prove \eqref{cocycle} we first make several observations.

Let $n$ a  non-negative integer. A case by case argument shows that if $0\le z_1,z_2,z_3< n$ then
\begin{equation}\label{TresCocientes}
\e{\frac{[z_1+z_2]_n+z_3}{n}}+ \e{\frac{z_1+z_2}{n}} = 
\e{\frac{z_1+[z_2+z_3]_n}{n}}+ \e{\frac{z_2+z_3}{n}}.
\end{equation}
We also observe that
$$\Ese{r_i}{n} = \sum_{j=0}^{\e{\frac{n}{N_i}}-1} r_i^{jN_i} \sum_{k=0}^{N_i-1} r_i^k + 
r_i^{N_i\e{\frac{n}{N_i}}} \sum_{k=0}^{[n]_{N_i}} r_i^k$$ 
and hence  using \eqref{Relacion-rN} we obtain
\begin{equation}\label{Relacion-rn}
\Ese{r_i}{n} \equiv 
\e{\frac{n}{N_i}} \Ese{r_i}{N_i} + \Ese{r_i}{[n]_{N_i}} \bmod M.
\end{equation}
Arguing by induction on $n$, congruences \eqref{Relacion-t1r2} and \eqref{Relacion-t2r1} generalize to
\begin{eqnarray}
\label{t1r2n} t_1 & \equiv & t_1r_2^n + \Ese{r_2}{n} \Ese{r_1}{N_1} \bmod M, \\
\label{t2r1n} t_2 & \equiv & t_2r_1^n - \Ese{r_1}{n} \Ese{r_2}{N_2} \bmod M.
\end{eqnarray}
 
Let 
\begin{align*}
R &=  {  r_2^{y_3 } \Ese{r_1}{x_3} \Ese{r_2}{y_2}}+ {r_2^{y_2+y_3} \Ese{r_2}{y_1}}\Ese{r_1}{[x_2+x_3]_{N_1}}; \\
R' &= { r_1^{x_3}    r_2^{y_2+y_3} \Ese{r_1}{x_2} \Ese{r_2}{y_1}}+{  r_2^{y_3} \Ese{r_1}{x_3} \Ese{r_2}{[y_1+y_2]_ {N_2}}}; \\ 
T_1&= t_1 r_2^{y_2+y_3}\left( {  r_2^{ y_1}  \e{ \frac{x_1+[x_2+x_3]_{N_1}}{N_1}} } +   {  \e{ \frac{x_2+x_3}{N_1}}}\right); 
\\
T_1'&= t_1 r_2^{y_1+y_2+y_3}\left(  {    \e{   \frac{[x_1+x_2]_{N_1}+x_3}{N_1}    }} + {      \e{ \frac{x_1+x_2}{N_1} }} \right);
\\
T_2 &= t_2\left( {   \e{  \frac{y_1+[y_2+y_3]_{N_2}}{N_2}  } }+{  \e{   \frac{y_2+y_3}{N_2}}} \right);  \\
T_2'&= t_2 \left({  \e{  \frac{[y_1+y_2]_{N_2}+y_3}{N_2}  }}+ { r_1^{x_3}   \e{   \frac{y_1+y_2}{N_2} } } \right).
\end{align*}
Then \eqref{TresCocientes},  \eqref{t1r2n} and \eqref{Relacion-rn} imply
\begin{align*}
T_1'&= t_1 r_2^{y_1+y_2+y_3} \left( {   \e{ \frac{x_1+[x_2+x_3]_{N_1}}{N_1}} } +   {  \e{ \frac{x_2+x_3}{N_1}}}\right) \\
&= T_1  + t_1 r_2^{y_2+y_3} (r_2^{y_1}-1) {  \e{ \frac{x_2+x_3}{N_1}}}
\equiv T_1 - r_2^{y_2+y_3} \Ese{r_2}{y_1}  \Ese{r_1}{N_1}{  \e{ \frac{x_2+x_3}{N_1}}} \\
&\equiv T_1 -r_2^{y_2+y_3} \Ese{r_2}{y_1}    \Ese{r_1}{x_2+x_3} + r_2^{y_2+y_3} \Ese{r_2}{y_1} \Ese{r_1}{[x_2+x_3]_{N_1}} \bmod M.
\end{align*}
Similarly, \eqref{TresCocientes},  \eqref{t2r1n} and \eqref{Relacion-rn} imply
\begin{align*}
T_2' & \equiv T_2 + r_2^{y_3}\Ese{r_1}{x_3} \Ese{r_2}{y_1+y_2} - r_2^{y_3}\Ese{r_1}{x_3}\Ese{r_2}{[y_1+y_2]_{N_2}} \bmod M.
\end{align*}
 Therefore\begin{align*}
T_1'+T_2' +R'-R \equiv
& \; T_1 -r_2^{y_2+y_3} \Ese{r_2}{y_1}    \Ese{r_1}{x_2+x_3}
 +T_2 + r_2^{y_3}\Ese{r_1}{x_3} \Ese{r_2}{y_1+y_2} \\
& + {r_1^{x_3}    r_2^{y_2+y_3} \Ese{r_1}{x_2} \Ese{r_2}{y_1}}-{r_2^{y_3} \Ese{r_1}{x_3} \Ese{r_2}{y_2}}   \\
\equiv
& \; T_1 - r_2^{y_2+y_3} \Ese{r_2}{y_1} \Ese{r_1}{x_3} -  r_2^{y_2+y_3} r_1^{x_3} \Ese{r_2}{y_1} \Ese{r_1}{x_2} \\
&+T_2 +  r_2^{y_3} \Ese{r_1}{x_3} \Ese{r_2}{y_2} +r_2^{y_2+y_3}\Ese{r_1}{x_3} \Ese{r_2}{y_1 }\\
& + {r_1^{x_3}    r_2^{y_2+y_3} \Ese{r_1}{x_2} \Ese{r_2}{y_1}}-{r_2^{y_2} \Ese{r_1}{x_2} \Ese{r_2}{y_3}}  \\
= & \; T_1+T_2  \bmod M,
\end{align*}  
  which implies \eqref{cocycle}. 
 
\end{proof}

We are ready to finish the proof of the theorem.
Let $I=(p,m,n_1,n_2, \sigma_1,\sigma_2,  o_1,o_2,o'_1,o'_2,u_1,u_2)$ satisfy conditions \eqref{1}-\eqref{7} in the theorem.
We have to prove that   $I=\inv(G)$  for some $2$-generated non-abelian cyclic-by-abelian finite $p$-group  $G$. 
Let $r_1,r_2$ be as in \eqref{def-erres} and $t_i=u_ip^{m-o'_i}$ for $i=1,2$. 
Then,  by conditions \eqref{2}, \eqref{4}, \eqref{6b} and \eqref{7}, relations \eqref{Relacion-rN}-\eqref{Relacion-t2r1} hold  for $M=p^m$ and $N_i=p^{n_i}$. With the notation of \Cref{Cociclo}, consider the group extension
	$$1\rightarrow A \rightarrow G \rightarrow B \rightarrow 1$$
realizing the action $\sigma$ and the 2-cocycle $\rho$. 
That is $G=\GEN{a,b_1,b_2}$ and the following relations holds $a^{b_i}=a^{r_i}$, $[  b_2,b_1]=a$  and $b_i^{p^{n_i}}=a^{t_i}$ because $\rho(b_2,b_1)=a$, $\rho(b_i,b_i^{k-1})=1$ if $k<p^{n_i}$ and $\rho(b_i,b_i^{p^{n_i}-1})=a^{t_i}$.
This implies that $G$ is the group given by the presentation in \eqref{Presentacion}. 
 From now on the notation for $\B$ and its variants refers to this group.  
  Observe that  $b=(b_1,b_2)\in\B$,  $\sigma(b_i)=\sigma_i$ and 
$o(b_i)=o_i$ for $i=1,2$ by the definition of the $r_i$'s.   Moreover, $o'(b)=(o'_1,o'_2)$ and $u(b)=(u_2,u_1)$,  by the definition of the $t_i$'s.   
  Moreover, $b$ satisfies the conditions in   \Cref{Fijando-rLema}, as the parameters satisfy the conditions \eqref{6a},  \eqref{7b}  and \eqref{5} of the theorem. Then $b\in \B'$ and therefore $ \sigma o(G)=(\sigma_1,\sigma_2,o_1,o_2)$. Hence $b\in \B_r$  by the definitions of $r_i$ and $t_i$.   
Using now \Cref{OPrima++} if $\sigma_1=1$, or \Cref{OPrima-+} and \Cref{OPrima--} if $\sigma_1=-1$, it follows that  $b\in \B'_r$, since the parameters satisfy conditions  \eqref{6c}, \eqref{6d} and \eqref{6e}, if $\sigma_1=1$, and otherwise they satisfy conditions     \eqref{7b} and \eqref{7c}.
Finally, \Cref{ues+} and \Cref{ues-} yield $b\in \B_{rt}$, since the parameters satisfy conditions   \eqref{6f}, \eqref{6g}, \eqref{7b} and \eqref{7c}.
Therefore $\inv(G)=I$, 
as desired. 

\section{Implementation}\label{SectionImplementation} 

In this section we present some {\sf GAP} \cite{GAP4} functions dealing with $2$-generated cyclic-by-abelian finite $p$-groups. The code of these function is available in \verb+https://www.um.es/adelrio/CbA2G.php+.


The function \verb+CBA2GenByOrder(p,n)+ provides the list formed by all the 12-tuples
$\inv(G)$ with $G$ a non-abelian cyclic-by-abelian 2-generated group of order $p^n$. For example, the following calculation shows that there are exactly 273 and 100 isomorphism classes of such groups of order $2^{10}$ and $3^{10}$, respectively.

\medskip

\begin{verbatim}
gap> l1:=CbA2GenByOrder(2,10);;l2:=CbA2GenByOrder(3,10);;
gap> Length(l1);Length(l2);
273
100
\end{verbatim}
\medskip
We use the output of the previous computation to pick the groups $G$ and $H$ with $\inv(G)$ and $\inv(H)$ as follows:
\medskip
\begin{verbatim}
gap> l1[210];l2[92];                                                      
[ 2, 5, 3, 2, -1, 1, 0, 1, 1, 4, 1, 7 ]
[ 3, 3, 5, 2, 1, 1, 2, 1, 1, 2, 2, 1 ]
\end{verbatim}
\medskip
By \eqref{def-erres}, \eqref{def-tes} and \eqref{Presentacion}
	$$G=\GEN{b_1,b_2 \mid a=[b_2,b_1], a^{2^5}=1, a^{b_1}=a^{-1}, a^{b_2}=a^{1+2^4}, b_1^{2^3}=a^{2^4}, b_2^{2^2}=a^{7\cdot 2}}$$
and
	$$H=\GEN{b_1,b_2 \mid a=[b_2,b_1], a^{3^3}=1, a^{b_1}=a^{1+3}, a^{b_2}=a^{12}, b_1^{3^5}=a^{2\cdot 3^2}, b_2^{3^2}=a^{2\cdot 3}}.$$ 

The function \verb+CBA2GenPcp(x)+ implements in {\sf GAP} the group $G$ with $x=\inv(G)$.

\medskip

\begin{verbatim}
gap> G:=CbA2GenPcp(l1[210]);DG:=DerivedSubgroup(G);;Order(DG);                  
<pc group of size 1024 with 10 generators>
32
gap> AbelianInvariants(G);NilpotencyClassOfGroup(G);                           
[ 4, 8 ]
6
gap> H:=CbA2GenPcp(l2[92]);DH:=DerivedSubgroup(H);;Order(DH);
<pc group of size 59049 with 10 generators>
27
gap> AbelianInvariants(H);NilpotencyClassOfGroup(H);         
[ 9, 243 ]
4
gap> StructureDescription(G);                                               
"C8 . ((C32 x C2) : C2) = C32 . (C8 x C4)"
gap> StructureDescription(H);
"C81 . (C27 : C27) = C27 . (C243 x C9)"
\end{verbatim}

\medskip

The function \verb+InvariantsAndBasis(G)+, with input a cyclic-by-abelian 2-generated finite $p$-group, outputs a pair $(\inv(G),[b_1,b_2])$ where $[b_1,b_2]$ is an element of the set $\B_{rt}$ of $G$, i.e. $G=\GEN{b_1,b_2}$ and $b_1$ and $b_2$ satisfy the relations of  presentation \eqref{Presentacion}. The function \verb+Invariants(G)+ only outputs the list $\inv(G)$. 

\medskip

\begin{verbatim}
gap> ib:=InvariantsAndBasis(G);
[ [ 2, 5, 3, 2, -1, 1, 0, 1, 1, 4, 1, 7 ], [ f1*f3*f5, f4 ] ]
gap> b1:=ib[2][1];; b2:=ib[2][2];; a := Comm(b2,b1);;
gap> Order(a);
32
gap> a^b1=a^-1 and a^b2=a^(1+2^4) and b1^(2^3)=a^(2^4) and b2^(2^2)=a^(7*2);
true
gap> Invariants(H);
[ 3, 3, 5, 2, 1, 1, 2, 1, 1, 2, 2, 1 ]
\end{verbatim}

\medskip

The function \verb+AreIsomorphicGroups(G,H)+ for $G$ and $H$ non-abelian 2-generated cyclic-by-abelian $p$-groups checks whether $G$ and $H$ are isomorphic by comparing $\inv(G)$ and $\inv(H)$. 

The function \verb+IsomorphismCbAGroups(G,H)+ provides an isomorphism between $G$ and $H$ in case it exists. 

\medskip

\begin{verbatim}
gap> G:=SmallGroup(2^8,465);
<pc group of size 256 with 8 generators>
gap> inv:=Invariants(G);
[ 2, 4, 2, 2, -1, -1, 0, 1, 0, 0, 1, 1 ]
gap> H:=CbA2GenPcp(inv);
<pc group of size 256 with 8 generators>
gap> AreIsomorphicGroups(G,H);
true
gap> IsomorphismCbAGroups(G,H);
[ f1, f1*f2 ] -> [ f1, f3 ]
gap> K:=SmallGroup(2^8,532);
<pc group of size 256 with 8 generators>
gap> AreIsomorphicGroups(K,H);
false
gap> Invariants(K);
[ 2, 2, 5, 1, -1, 1, 0, 0, 1, 2, 1, 1 ]
\end{verbatim}

\medskip

The function \verb+DescendantsCbA2Gen(p,n)+ computes representatives of the isomorphism  classes of non-abelian cyclic-by-abelian 2-generated groups of order $p^n$, using the $p$-group generation algorithm \cite{OBrien1990} as implemented for {\sf GAP} in \cite{ANUPQ}.

\medskip

\begin{verbatim}
gap> x:=DescendantsCbA2Gen(2,10);;
gap> Length(x);
273
\end{verbatim}

\medskip

The functions \verb+CheckNumber(p,n)+ and \verb+CheckIsoClasses(p,n)+ compares the outputs of \verb+CbA2GenByOrder(p,n)+ and \verb+DescendantsCbA2Gen(p,n)+ returning \verb+true+ in case the outputs agrees. More precisely, \verb+CheckNumber(p,n)+ returns true if the outputs of \verb+CbA2GenByOrder(p,n)+ and \verb+DescendantsCbA2Gen(p,n)+ have the same cardinality, as they should. 
The output of \verb+CheckIsoClasses(p,n)+ is \verb+true+ when the list obtained by applying \verb+Invariants+ to the list of groups given by \verb+DescendantsCbA2Gen(p,n)+ coincides, maybe in different order, with the output of \verb+CbA2GenByOrder(n,p)+.

The following calculation provides an experimental verification of the correctness of the theorem up to our computational capacity. It is a costly calculation, which for large values requires allowing 4Gb or even 8Gb of memory. We have been able to verify that the output is \verb+true+ up to the following orders: $2^{12}, 3^{11}, 5^{10}, 7^9, 11^8, 13^7$ and $23^8$.

\medskip

\begin{verbatim}
gap> CheckIsoClasses(2,12);
true
gap> CheckIsoClasses(3,10);
true
\end{verbatim}

 \section{Appendix:  The operators $\Ese{-}{-}$ and $\Te{-,-}{-}$} \label{Apendice}
 
This section is dedicated to prove some   useful properties of the operators $\Ese{-}{-}$ and $\Te{-,-}{-}$ defined at the beginning of \Cref{SectionBr}.
 
 \begin{lemma}\label{PropEse}
 	If $x,y$ are integers and $a,b,c$ and $d$ are positive integers then 
 	\begin{enumerate}
 		\item $\Ese{x}{a} = \begin{cases} a, & \text{if } x=1; \\ \frac{x^a-1}{x-1}, & \text{otherwise}.\end{cases}$
 		\item $\Ese{x}{1+a}=1+x\Ese{x}{a}$.
 		\item $\Ese{x}{ab} = \Ese{x}{a}\Ese{x^a}{b}$. 
 		\item\label{EseCon1} $(x-1)\Te{x,1}{a} = \Ese{x}{a}-a$.
 	\end{enumerate}
 \end{lemma}
 
 \begin{proof}
 	The first three properties are obvious. The fourth one follows from the following calculation
 	$$ {(x-1)}\Te{x,1}{a} = (x-1)\sum_{j=1}^{a-1} \sum_{i=0}^{j-1} x^i = \sum_{j=1}^{a-1} (x^j-1) = \Ese{x}{a}-a.$$
 \end{proof}
 
 \begin{lemma}\label{EseProp}
 	Let $p$ be a prime integer and let $s$ and $n$ be integers with $n>0$ and $s\equiv 1 \bmod p$.
 	\begin{enumerate}
 		\item\label{ValEse} If  $p$ is odd or   $s\equiv 1 \bmod 4$ then $v_p(s^n-1)=v_p(s-1)+v_p(n)$. Therefore $v_p(\Ese{s}{n})=v_p(n)$ and $\ord_{p^n}(s)={p^{\max(0,n-v_p(s-1))}}$.
 		\item\label{ValEse-1} If $p=2$ and $s\equiv -1\bmod 4$  then
 		$$v_2(s^n-1)= \begin{cases}
 		1, & \text{if } {  2\nmid n}; \\
 		v_2(s+1)+v_p(n),& \text{otherwise{;}}
 		\end{cases}$$
 		$$v_2(\Ese{s}{n})=\begin{cases}  
 		0, & \text{if } {  2\nmid n};\\ 
 		v_2(n)+v_2(s+1)-1, &\text{otherwise{;}} \end{cases}$$  
 		and if $n\ge 2$ then   $\ord_{2^n} (s)= { 2}^{\max(1,n -v_2(s+1))} $.
 		
 		\item\label{EseCero} If $v_p(s-1)=a$ and either $n$ or $n-1$ is multiple of $p^{m-a}$ then
 		$$\Ese{s}{n} \equiv \begin{cases}
 		0 \bmod 2^m, & \text{if } p=2, a=1<m, \text{ and } n\equiv 0 \bmod 2^{m-1}; \\
 		1 \bmod 2^m, & \text{if } p=2, a=1<m, \text{ and } n\equiv 1 \bmod 2^{m-1}; \\
 		n+2^{m-1} \bmod 2^m, & \text{if } p=2, 2\le a<m \text{ and }  n \not\equiv 0,1 \bmod 2^{m-a+1}; \\
 		n \bmod p^m, & \text{otherwise}.
 		\end{cases}$$
 	\end{enumerate}
 \end{lemma}
 
 \begin{proof}
 	\eqref{ValEse} is clear if $s=1$ (with the convention that $\infty+n=\infty$ and $n-\infty=-\infty<0$) so suppose that $s\ne 1$. Then $\Ese{s}{n} = \frac{s^n-1}{s-1}$,  {and hence} $v_p(\Ese{s}{n}) = v_p(s^n-1)-v_p(s-1)$ so to prove \eqref{ValEse} by induction on $v_p(n)$ it is enough to show that if $p\nmid n$ then $v_p(s^n-1)=v_p(s-1)$ and $v_p(s^p-1)=v_p(s-1)+1$.  Let $a=v_p(s-1)$. So $s=1+kp^a$ with $p\nmid k$.  	Then $s^n = 1+knp^a + \sum_{i=2}^n \binom{n}{i} k^i p^{ia}\equiv 1+knp^a \bmod p^{a+1}$.  Thus, if $p\nmid n$ then $v_p(s^n-1)=a$. Moreover
$s^p = 1+kp^{a+1} + \sum_{i=2}^p \binom{p}{i} k^i p^{ia}$. If $2\le i 
<p$ then $v_p(\binom{p}{i} p^{ia})=1+ap\ge a+2$.  {In particular,}  if $p\ne 2$ then $v_p(s^p-1)=a+1$. If  $p=2$   then  $s\equiv 1 \bmod 4$,  by hypothesis.   Therefore   $a\ge 2$ and hence $2a>a+1$ and   $s^2 = 1+k2^{a+1}+k^22^{2a}\equiv 1+k2^{a+1}\bmod 2^{a+2}$.   {So, in both cases $v_p(s^p-1)=a+1$, as desired.}

 	The hypothesis $s\equiv 1 \bmod p$ implies that $\ord_{p^n}(s)= p^m$ for
 	\begin{align*}
 	m &=\min \{ i\geq 0 \ : \ s^{p^i}\equiv 1\bmod p^n  \} =   \min \{ i\geq 0 \ : \ v_p(s^{p^i}-1)\geq n  \} \\
 	&=\min \{ i\geq 0 \ : \ i+ v_p(s-1)  \geq n \}     = \max(0,n-v_p(s-1)).
 	\end{align*}

 	\eqref{ValEse-1}   The argument above shows in general that $v_p(s^n-1)=v_p(s-1)$ if $p\nmid n$. In particular if $2\nmid n$ then   $v_2(s^n-1)=v_2(s-1)=1$, and consequently $v_2(\Ese{s}{n})= v_2(s^n-1)-v_2(s-1)=0$, because by assumption $s\equiv -1 \bmod 4$.
 	Since $s^2\equiv 1 \mod 4$, if $2\mid n$ then \eqref{ValEse} yields $v_2(s^n-1)=v_2((s^2)^{\frac{n}{2}}-1) = v_2(s^2+1)+v_2\left(\frac{n}{2}\right) = v_2(s+1)+v_2(s-1)+v_2(n)-1=v_2(s+1)+v_2(n)$.
 	The assertions about   $v_2(\Ese{s}{n})$ and $\ord_{2^n}(s)$ follows as
in   the proof of  \eqref{ValEse}.  
 	
 	\eqref{EseCero} The statement is clear if $a\ge m$ so we assume that $a<m$. 
 	
 	Suppose first that either $p$ is odd or $a\ge 2$. In that case, by \eqref{ValEse}, $\ord_{p^m}(s)=p^{m-a}$ and thus the multiplicative
group $\GEN{s}$ generated by $s$ in $\Z /{p^m\Z}$ is formed by the classes represented by the integers of the form $1+ip^a$ with $0\le i < p^{m-a}$. 
 	Thus 
 	\begin{eqnarray*}
 		\Ese{s}{p^{m-a}} & \equiv & \sum_{i=0}^{p^{m-a}-1} (1+ip^a) = p^{m-a}+p^a \sum_{i=0}^{p^{m-a}-1} i 
 		= p^{m-a}+p^a \frac{(p^{m-a}-1)p^{m-a}}{2} \\
 		&=& p^{m-a}+\frac{(p^{m-a}-1)p^m}{2} \equiv \begin{cases} p^{m-a} \bmod p^m, & \text{if } p\ne 2; \\
 			2^{m-a}+2^{m-1} \bmod 2^m, & \text{ otherwise}.\end{cases}
 	\end{eqnarray*}
 	Now using that $s^i \equiv s^j \bmod p^m$ if $i\equiv j \bmod p^{m-a}$ we
deduce that $\Ese{s}{bp^{m-a}} \equiv b \Ese{s}{p^{m-a}} \bmod p^m$ for each integer $b$.
 	Therefore, if $p^{m-a}$ divides $n$ then 
 	$$\Ese{s}{n} \equiv \frac{n}{p^{m-a}} \Ese{s}{p^{m-a}} \equiv \begin{cases} 
 	n  \bmod p^m, & \text{if } p\ne 2; \\
 	n+\frac{n}{2^{m-a}} 2^{m-1}  \bmod 2^m, & \text{if } p= 2. \end{cases}$$
 	or equivalently if $p^{m-a}$ divides $n$ then 
 	$$  \Ese{s}{n} \equiv \begin{cases}
 	n+2^{m-1} \bmod 2^m, & \text{if } p=2 \text{ and } n \not\equiv  0 \bmod 2^{m-a+1}; \\
 	n \bmod p^m, & \text{otherwise}.
 	\end{cases}$$
 	
 	If $n\equiv 1 \bmod p^{m-a}$ then $s^{n-1}\equiv 1 \bmod p^m$, and the previous statement for $n-1$ yields
 	$$\Ese{s}{n} = \Ese{s}{n-1} + s^{n-1} \equiv 
 	\begin{cases}
 	n+2^{m-1} \bmod 2^m, & \text{if } p=2 \text{ and } n \not\equiv 1 \bmod
2^{m-a+1}; \\ 
 	n \bmod p^m, & \text{otherwise}.
 	\end{cases}$$
 	
Now consider the case $p=2$ and $a=1$.
 	First of all observe that if $m\ge 2$ then  $\Ese{s}{2^{m-1}} = (1+s)\Ese{s^2}{2^{m-2}}$. 
 	Let $b=v_2(s+1)$. As $a=1$, $b\ge 2$ and $v_2(s^2-1)=b+1\ge 3$.
 	Applying the results above for $s^2$ and $ {  2}^{m-2}$ in the roles of 
$s$ and $n$ respectively we deduce that $\Ese{s^2}{2^{m-2}}\equiv 2^{m-2} 
\bmod 2^m$. In particular, $2^{m-2}$ divides $\Ese{s^2}{2^{m-2}}$. As $4\mid s+1$
we deduce that $\Ese{s}{2^{m-1}}\equiv 0 \bmod 2^m$.
 	Arguing as above we deduce that if $2^{m-1}$ divides $n$ then $\Ese{s}{n}\equiv 0 \bmod 2^m$.
 	Applying this to $n-1$ we deduce that if $n\equiv 1\bmod 2^{m-1}$ then
 $s^{n-1}\equiv 1 \bmod 2^m$ and  $\Ese{s}{n-1}\equiv 0 \bmod 2^m$.
 	Hence $\Ese{s}{n} = \Ese{s}{n-1}+s^{n-1} \equiv 1 \bmod 2^m$.

 \end{proof}
 
 In the proof of the following two lemmas we will use the following equality:
 \begin{equation}\label{eqdobleese}
 \begin{split}
 \Te{s,t}{p^{n+1}} &= \sum_{0\le i<j<p^{n+1}} s^it^j =
 \sum_{k=0}^{p-1} \sum_{\stackrel{kp^n\le i < (k+1)p^n,}{i< j < p^{n+1}}} s^i t^j \\
  &=
  \sum_{k=0}^{p-1}
  \left( \sum_{kp^n\le i < j < (k+1)p^n} s^i t^j +
  \sum_{\stackrel{kp^n\le i < (k+1)p^n,}{(k+1)p^n\le j < p^{n+1}}} s^i t^j
 \right) \\
  &=
  \sum_{k=0}^{p-1}
  \left( s^{kp^n} t^{kp^n} \sum_{0\le i < j < p^n} s^i t^j +
  s^{kp^n} t^{(k+1)p^n}\sum_{0\le i < p^n, 0\le j < p^n(p-k-1)} s^i t^j\right) \\
  & =
  \Ese{s^{p^n}t^{p^n}}{p} \Te{s,t}{p^n} +
  t^{p^n}\Ese{s}{p^n}\sum_{k=0}^{p-1} s^{kp^n} t^{kp^n} \Ese{t}{p^n(p-k-1)}.
 \end{split}
 \end{equation}

 \begin{lemma}\label{ValEse2}
 	Suppose that $s,t\equiv 1 \bmod p$ and $n$ is a positive integer.
 	Then 
 	$$\Te{s,t}{p^n} \equiv \begin{cases} 0 \bmod p^n, & \text{if } p\ne 2; \\
 	2^{n-1} \bmod 2^n, & \text{if } p=2.
 	\end{cases}$$
 \end{lemma}
 
 \begin{proof}
We argue by induction on $n$ with the case $n=1$ being obvious. Suppose that the statement holds for $n$. Observe that $s^{p^n} \equiv t^{p^n}
\equiv 1 \bmod p^{n+1}$.
Moreover, by Lemma~\ref{EseProp}.\eqref{EseCero}, $\Ese{s}{p^n}\equiv
\Ese{t}{ip^n }\equiv 0 \bmod p^n$. Hence, by \eqref{eqdobleese}, $\Te{s,t}{p^{n+1}}\equiv p \Te{s,t}{p^n} \bmod p^{n+1}$.
 	By induction hypothesis if $p\ne 2$ then $\Te{s,t}{p^n}$ is multiple of $p^n$ and hence $\Te{s,t}{p^{n+1}}\equiv 0 \bmod p^{n+1}$.
 	If $p=2$ then $\Te{s,t}{2^n}= 2^{n-1}+a2^n $ for some integer $a$ and hence  $\Te{s,t}{ { 2}^{n+1}}\equiv 2^n \bmod 2^{n+1}$.
 \end{proof}

 \begin{lemma}\label{ValEse2Rebuscado}
 	Let $m$ be a positive integer, and let  $s_1,s_2$  be  integers such that $s_1\equiv -1\bmod 4$ and $s_2\equiv 1 \bmod 2$.   Denote  $o_1= \max(0,m-v_2(s_1+1))$ and  $o_2=\max(0, m-v_2(s_2-1))$.    If $n$ is a positive integer such that  $ \max(o_1,o_2) \leq n-1$ then
 	$\Te{s_1,s_2}{2^n} \equiv 2^{n-1} \bmod 2^m. $
 \end{lemma}
 \begin{proof}
We proceed by double  induction, first on $m$ and then on $n $. As  $s_1\equiv s_2\equiv 1\bmod 2$  we have   $\Te{s_1,s_2}{2^n}\equiv \Te{1,1}{2^n} \equiv 2^{n-1} \bmod 2 $ for every   $n$. Now assume  that  $m\ge 2$
and the induction hypothesis for $m-1$ and proceed  by induction over $n$. If $n=1$ then $o_1=o_2=0$, so $s_1 \equiv -1 \bmod 2^m$ and $s_2 \equiv 1\bmod 2^m$, and  hence  $ \Te{s_1,s_2}{2} \equiv \Te{-1,1}{2} = 1 \bmod 2^m$.   Assume  that $n\ge 2$ and the induction hypothesis for $n-1$. Observe that the hypothesis $\max(o_1,o_2)\le n-1$, combined with \Cref{EseProp}, implies that $s_1^{2^{n-1}}\equiv s_2^{2^{n-1}}\equiv 1 \bmod 2^m$.
This and \eqref{eqdobleese} yield
 	$$\Te{s_1,s_2}{2^{n }} \equiv
 	2\Te{s_1,s_2}{2^{n-1}} + \Ese{s_1}{2^{n-1}}\Ese{s_2}{2^{n-1}} \bmod 2^m.$$
By Lemma~\ref{EseProp},  $v_2(\Ese{s_1}{2^{n-1}}\Ese{s_2}{2^{n-1}} )\geq v_2(\Ese{s_1}{2^{n-1}})+1 = n -1 +v_2(s_1+1) \geq   n-1+m-o_1$, which is greater or equal than $m$ by hypothesis.
 	Hence $\Te{s_1,s_2}{2^{n }}\equiv 2 \Te{s_1,s_2}{2^{n-1}}  \bmod 2^m$.
If $\max(o_1,o_2)<n-1$ then  we can apply the induction hypothesis (on $n$) to deduce that $\Te{s_1,s_2}{2^{n-1}} \equiv 2^{n-2} \bmod 2^m$, and the result follows. Thus we can assume $\max(o_1,o_2)=n-1$.
 	Then write $\tilde m=m-1$, $\tilde o_1=\max(0, \tilde m- v_2(s_1+1))= {\max(0,o_1-1)}$ and $\tilde o_2=\max(0, \tilde{m} -v_2(s_2-1))= 
{\max(0,o_2-1)}$. 
 	As $\max(o_1,o_2)=n-1  \geq  1$, $\max(\tilde o_1,\tilde o_2) =n-2 $, so by the induction hypothesis (on $m$) $\Te{s_1,s_2}{2^{n-1}} \equiv 2^{n-2} \bmod 2^{m-1}$.  Hence $\Te{s_1,s_2}{2^n}\equiv 2\Te{s_1,s_2}{2^{n-1}}\equiv 2^{n-1} \bmod 2^m$.
 \end{proof}

\bibliographystyle{amsalpha}
\bibliography{ReferencesMSC}

\end{document}